%% file: lru_mat_equations.tex
\documentclass{siamart1116}

\usepackage{amsmath,amssymb}
\usepackage{cite}
\usepackage{mathtools}
\usepackage{tikz}
\usetikzlibrary{patterns}
\usepackage{pgfplots}
\usepackage{pgfplotstable}
\usepackage{geometry}
\usepackage{color}
\usepackage{booktabs}
\usepackage{mathrsfs}
\usepackage{cleveref}
\usepackage[utf8]{inputenc}
\usepackage{enumitem}
\usepackage{algorithm}
\usepackage{algpseudocode}
\newcommand\TheTitle{Low-rank updates and a divide-and-conquer method for linear matrix equations}

\newcommand\TheAuthors{Daniel Kressner, 
	Stefano Massei, and Leonardo Robol}
\newcommand{\trid}{\hbox{trid}}

\headers{Low-rank updates and divide-and-conquer for linear matrix equations}{\TheAuthors}

\title{\TheTitle\thanks{The research of the third author was partially supported by GNCS projects ``Metodi numerici
		avanzati per equazioni e funzioni di matrici con struttura'' and ``Tecniche innovative per
		problemi di algebra lineare''.}}

\author{
	Daniel Kressner\thanks{EPF Lausanne, Switzerland,
		\email{daniel.kressner@epfl.ch}} \and
	Stefano Massei\thanks{EPF Lausanne, Switzerland,
		\email{stefano.massei@epfl.ch}} \and
              Leonardo Robol\thanks{Department of Mathematics, University of Pisa,
                and ISTI-CNR, Pisa, Italy, 
                \email{leonardo.robol@unipi.it}}
}

\pgfplotsset{compat=1.9}

\pgfplotstableset{
	every head row/.style={before row=\toprule,after row=\midrule},
	clear infinite
}

\usepackage{amsopn}
\DeclarePairedDelimiter{\norm}{\lVert}{\rVert}
\DeclarePairedDelimiter{\abs}{\lvert}{\rvert}

\DeclareMathOperator{\Span}{span}

\DeclareMathOperator{\capac}{Cap}

\numberwithin{theorem}{section}

\newsiamremark{remark}{Remark}
\newsiamremark{example}{Example}

\renewcommand{\leq}{\leqslant}

\ifpdf
\hypersetup{
	pdftitle={\TheTitle},
	pdfauthor={\TheAuthors}
}
\fi

\begin{document}
	\maketitle
	\input{abstract}
	\input{section1}
\input{section2}
	\input{section3}
	\input{section4}\input{section5}
	\input{conclusions}

\bibliography{anchp,lru}
\bibliographystyle{siamplain}

	\end{document}

%% file: abstract.tex
\begin{abstract}
Linear matrix equations, such as the Sylvester and Lyapunov equations, play an important role in various applications, including the stability analysis and dimensionality reduction of linear dynamical control systems and the solution of partial differential equations. In this work, we present and analyze a new algorithm, based on tensorized Krylov subspaces, for quickly updating the solution of such a matrix equation when its coefficients undergo low-rank changes. We demonstrate how our algorithm can be utilized to accelerate the Newton method for solving continuous-time algebraic Riccati equations. Our algorithm also forms the basis of a new divide-and-conquer approach for linear matrix equations with coefficients that feature hierarchical low-rank structure, such as HODLR, HSS, and banded matrices. Numerical experiments demonstrate the advantages of divide-and-conquer over existing approaches, in terms of computational time and memory consumption.
\bigskip


\bigskip 

\end{abstract}

%% file: section1.tex
	\section{Introduction}
	This work is concerned with linear matrix equations of the form
\begin{equation}\label{eq:sylv}
	AX+XB=C,
\end{equation}
for given matrices $A\in\mathbb C^{n\times n}$, $B\in\mathbb C^{m\times m}$, and $C\in\mathbb C^{n\times m}$. It is well known that this equation admits a unique solution $X$ if and only if the spectra of $A$ and $-B$ are disjoint. For general coefficient matrices,~\eqref{eq:sylv} is usually called Sylvester equation. In the special case $B = A^*$ and $C=C^*$,~\eqref{eq:sylv} is called Lyapunov equation and its solution can be chosen Hermitian. If, moreover, $C$ is negative semi-definite and $A$ is stable (i.e., its spectrum is contained in the open left half plane) then the solution is positive semi-definite.

We specifically target the setting where both $m,n$ are large and $A,B,C$ admit certain data-sparse representations, such as sparsity or (hierarchical) low-rank structures. The need for solving such large-scale linear matrix equations arises in various application fields. In dynamical systems and control, Lyapunov equations arise in model reduction~\cite{Antoulas2005}, linear-quadratic optimal control~\cite{Benner2013}, and stability analysis~\cite{MITbook,Elman2012}. In these applications, it is often but not always the case that  $C$ has low-rank.  Partial differential equations (PDEs) are a frequent source of  Sylvester equations, where they typically arise from highly structured discretizations of PDEs with separable coefficients; see  \cite{Grasedyck2007,Kressner2010,Ringh2017,Simoncini2016} for recent examples. Other applications arise from the linearization of nonlinear problems, such as stochastic dynamic general equilibrium models in macroeconmics~\cite{Kamenik2005}.

In this work, we study low-rank updates for Lyapunov and Sylvester equations. Given the solution $X_0$ of the reference equation
\begin{equation}\label{eq:reference-sylv}
A_0 X_0 + X_0 B_0 = C_0, 
\end{equation}
we aim at computing a correction $\delta X$ such that $X_0+\delta X$ solves the perturbed equation
\begin{equation}\label{eq:update-sylv}
	(A_0+\delta A)(X_0+\delta X)+(X_0+\delta X)(B_0+\delta B)=C_0+\delta C,
\end{equation}
where the perturbations $\delta A$, $\delta B$, $\delta C$ all have ranks much smaller than $\min\{m,n\}$. This is not only a natural problem to study but it also occurs in some applications. For example, it arises when optimizing dampers in mechanical models~\cite{Kuzmanovic2013} or, as we will see below, in the Newton method for solving Riccati equations. However, we expect, and it will be demonstrated in the second part of this paper, that the availability of a fast technique for computing $\delta X$ will open up a range of other applications.

The literature is scarce on updates of the form~\eqref{eq:update-sylv}. Kuzmanovi\'c and Truhar~\cite{Kuzmanovic2013} view the left-hand side~\eqref{eq:update-sylv} as a low-rank perturbation $\mathbb L + \triangle \mathbb L$ of the operator $\mathbb L: X_0 \to A_0 X_0 + X_0 B_0$. In turn, this allows to apply operator variants of the Sherman-Morrison-Woodbury formula discussed, e.g., in~\cite{Damm2008,Kuzmanovic2013,Richter1993}. This approach is mathematically equivalent to applying the standard Sherman-Morrison-Woodbury to the $n^2\times n^2$ linear system corresponding to~\eqref{eq:update-sylv} and it allows to deal with a much larger class of perturbations, leaving the realm of Sylvester equations. However, it also comes with the major disadvantage of increasing the ranks significantly. For example, if $\delta A$ has rank $r$ then the operator $X \to \delta A X$, with the matrix representation $I_n \otimes A$, has rank $rn$. This makes it impossible to address large values of $n$ with existing techniques for solving Sylvester equations.

The approach proposed in this work proceeds by subtracting equation~\eqref{eq:reference-sylv} from~\eqref{eq:update-sylv}, which gives the correction equation
\begin{equation}\label{eq:delta-sylv}
(A_0 + \delta A)\delta X+ \delta X (B_0+\delta B) = \delta C - \delta A\cdot X_0 -X_0\cdot \delta B.
\end{equation}
This is again a Sylvester equation, but -- in contrast to~\eqref{eq:update-sylv} --  the right-hand side always has low rank; it is bounded by $\text{rank}(\delta A) + \text{rank}(\delta B) + \text{rank}(\delta C)$. This allows for the use of large-scale Sylvester solvers tailored to low-rank right-hand sides, such as low-rank ADI and (rational) Krylov subspace methods;
 see~\cite{Benner2013,Simoncini2016} for overviews. These techniques return a low-rank approximation to $\delta X$ and can potentially address very large values of $m,n$, as long as the data-sparsity of $A_0 + \delta A$ and $B_0 + \delta B$ allows for fast matrix-vector multiplication and/or solution of (shifted) linear systems with these matrices. Let us emphasize that our approach is of little use when the rank of $C_0$ is at the same level as the ranks of the perturbations or even lower. In this case, it is more efficient to solve~\eqref{eq:update-sylv} directly. 

In the second part of this work, we devise fast methods for Sylvester equations with coefficients $A,B,C$ that feature hierarchical low-rank structures. In this work, we focus on HODLR matrices~\cite{Ambikasaran2013}, a special case of hierarchical matrices~\cite{Hackbusch2015}, and HSS matrices~\cite{Xia2010}, a special case of ${\mathcal H}^2$-matrices~\cite{Borm2010}. Both formats include banded matrices as an important special case. In fact, there has been recent work by Haber and Verhaegen~\cite{Haber2016} that aims at approximating the solution $X$ by a banded matrix for Lyapunov equations with banded coefficients. Palitta and Simoncini~\cite{Palitta2017} consider the approximation of $X$ by the sum of a banded matrix and a low-rank matrix. Both approaches work well for well-conditioned Lyapunov equations but their memory consumption grows considerably as the condition number increases. As we will see below, this difficulty is avoided when approximating $X$ with a hierarchical low-rank matrix instead, even when the coefficients are banded.

Most existing algorithms for solving Lyapunov or Sylvester equations with hierarchical low-rank structure are based on the matrix sign function iteration~\cite{Denman1976}, exploiting the fact that the iterates can be conveniently evaluated and approximated in these formats. The use of the matrix sign function requires the spectra of $A$ and $-B$ to be not only disjoint but 
separable by a (vertical) line. Sign-function based algorithms have been developed for hierarchical matrices~\cite{Grasedyck2003a,Baur2006}, sequentially semi-separable matrices~\cite{Rice2009}, HSS matrices~\cite{Pauli2010}, and HODLR matrices~\cite{Massei2017}.
Another, less explored direction is to apply numerical quadrature to an integral representation for $X$ and evaluate the resulting matrix inverses or exponentials in a hierarchical low-rank format; see~\cite{Gavrilyuk2005} for an example. All these methods exploit the structure indirectly by performing approximate arithmetic operations in the format. This incurs a repeated need for recompression, which often dominates the computational time.

In this work, we develop a new divide-and-conquer method that directly exploits hierarchical low-rank structure and does not require separability of the spectra of $A$ and $-B$. The main idea of the method is to split the coefficients $A,B,C$ into a block diagonal part and an off-diagonal part. The block diagonal part is processed recursively and the off-diagonal part, which is assumed to be of low rank, is incorporated by solving the correction equation~\eqref{eq:delta-sylv}.

The rest of this paper is organized as follows. In Section~\ref{sec:low-rank} we recall theoretical tools from the literature that ensure the low-rank approximability of the solution $\delta X$ of \eqref{eq:delta-sylv}. Section~\ref{sec:alg-and-ric} is devoted to describe in details the low-rank solver employed for approximating $\delta X$. We also discuss how to 
 rephrase the Newton's iteration, for solving  CAREs, as the updating of a  matrix equation. Numerical tests regarding this application are reported in Section~\ref{sec:ric-test}.
 In Section~\ref{sec:dac} we introduce a divide-and-conquer method for solving linear matrix equations whose coefficients can be hierarchically partitioned as block diagonal plus low-rank matrices. We provide an analysis in the case of coefficients represented in the HODLR and HSS formats. The algorithm is tested on examples coming from the discretization of PDEs and linear-quadratic control problems for second order models. The results are reported in Section~\ref{sec:num-res}. Finally, in Section~\ref{sec:conclusion} we draw the conclusions and we comment some open questions.

%% file: section2.tex
\section{Low-rank approximability} \label{sec:low-rank}

In this section, we recall existing results indicating when the correction $\delta X$ to the solution of the perturbed Sylvester equation~\eqref{eq:update-sylv} admits a good low-rank approximation. For this purpose, we write~\eqref{eq:delta-sylv} more compactly as
\begin{equation}  \label{eq:compactdelta}
  A\,\delta X + \delta X\, B = D, \qquad \text{rank}(D) \le
    s := \mathrm{rank}(A)+\mathrm{rank}(B) + \mathrm{rank}(C)
\end{equation}
In the following, we say that $\delta X$ admits an $\epsilon$-approximation of rank $k$ if there is a matrix $Y$ of rank at most $k$ such that $\|\delta X-Y\|_2\le \epsilon$, where $\|\cdot\|_2$ denotes the matrix
$2$-norm. Clearly, this is the case if and only if the $(k+1)$th largest
singular value $\sigma_{k+1}(\delta X)$ is bounded by $\epsilon$ (or the size of $\delta X$ is smaller than $k+1$).

There are numerous results in the literature on the singular value decay of solutions of equations of the form~\eqref{eq:compactdelta}; see~\cite{Antoulas2002,Baker2015,Grasedyck2003a,Grubisic2014,Penzl2000,Sabino2006} for examples. Recent work by Beckermann and Townsend~\cite{Beckermann2016} yields a general framework for obtaining such results. Let $\mathcal R_{h,h}$ denote the set of rational functions with numerator and denominator of degree at most $h$.  The proof of \cite[Thm 2.1]{Beckermann2016} shows that for every $r\in \mathcal R_{h,h}$ there is a matrix $Y_h$ of rank at most $kh$ such that $\delta X-Y_{h}= r(A)\, \delta X\, r(-B)^{-1}$, provided that the expression on the right-hand side is well defined. In turn,
\begin{equation} \label{eq:boundsigma}
 \sigma_{kh +1} (\delta X) \le \|r(A)\|_2 \|r(-B)^{-1}\|_2 \| X\|_2. 
\end{equation}
To proceed from here, we recall that the numerical range of a matrix $A$ is defined as
	\[
	\mathcal W(A) := \Big\{ \frac{x^* A x}{x^* x} \ \Big| \ x \in \mathbb {C}^n \backslash \{ 0 \} \Big\}. 
	\]
	
\begin{theorem} \label{thm:beckermann}
Consider the Sylvester equation~\eqref{eq:compactdelta} and let 
$E$ and $F$ be disjoint compact sets containing 
the numerical ranges of $A$ and $-B$, respectively. Then
\[
	\frac{\sigma_{1+kh}(\delta X)}{\norm{ X}_2}\leq  Z_h(E,F) := K_C\,\min_{r\in\mathcal R_{h,h}}\frac{\max_E \abs{r(z)}}{\min_F \abs{r(z)}},
\]
where $K_C = 1$ if $A,B$ are normal matrices and $1\le K_C \le (1+\sqrt{2})^2$ otherwise.
\end{theorem}
\begin{proof}
The result for normal matrices, which is also covered in~\cite{Beckermann2016}, follows immediately from~\eqref{eq:boundsigma} by diagonalizing $A,B$. To address the general case, we use Crouzeix's theorem~\cite{Crouzeix2017}, which implies
\begin{eqnarray*}
  \|r(A)\|_2 &\le& (1+\sqrt{2})\, \max_{z\in E} |r(z)|, \\
  \|r(-B)^{-1}\|_2 &\le& (1+\sqrt{2})\, \max_{z\in E} |1/r(z)| = (1+\sqrt{2}) \big( \min_{z\in F} |r(z)| \big)^{-1}.
\end{eqnarray*}
Combined with~\eqref{eq:boundsigma}, this completes the proof.
\end{proof}

The result of Theorem~\ref{thm:beckermann} links the (relative) singular values of $\delta X$ with the quantity $Z_{h}(E,F)$, usually called the 
$h$th Zolotarev number \cite{Beckermann2016}. Intuitively, this number becomes small when $E$ and $F$ are well separated.  The case when $E$, $F$ are intervals is particularly well understood and the considerations from~\cite[Sec. 3]{Beckermann2016} lead to the following result.
\begin{corollary} \label{cor:decay}
Let $A,B$ be Hermitian positive definite matrices with spectra contained in an interval $[a,b]$, $0<a<b<\infty$.
Then the solution $\delta X$ of~\eqref{eq:compactdelta} satisfies 
	\begin{equation}\label{eq:zol-decay}
	\frac{\sigma_{1+kh}(\delta X)}{\norm{\delta X}_2}\leq 4\rho^{-h},\qquad \rho:=\operatorname{exp}\left(\frac{\pi^2}{\log(4b/a)}\right).
	\end{equation}
\end{corollary}
Similar results have been derived in~\cite{Braess2005,Sabino2006}. The inequality~\eqref{eq:zol-decay} implies that $\delta X$ admits an $\epsilon$-approximation of rank $kh$ with $\epsilon$ exponentially decaying to zero as $h$ increases. Moreover, the relative separation $b/a$ of the spectra has a very mild logarithmic influence on the exponential decay rate.


Corollary~\ref{cor:decay} easily extends to the case of
diagonalizable coefficients with real spectra
\cite[Corollary~4.3]{Bini2017}. Letting $\kappa_{\mathrm{eig}}(A)$ and $\kappa_{\mathrm{eig}}(B)$ denote the condition numbers of eigenvector matrices of $A$ and $B$, respectively, one has
\[
\frac{\sigma_{1+kh}(\delta X)}{\norm{\delta X}_2}\leq 4 \kappa_{\mathrm{eig}}(A) \kappa_{\mathrm{eig}}(B)\rho^{-h}.
\]

When $E,F$ are not intervals, it appears to be difficult to derive bounds for $Z_h(E,F)$ that match~\eqref{eq:zol-decay} in terms of strength and elegance; see~\cite{Beckermann2016,lebailly} and the references therein.
Under reasonable assumptions, $Z_h(E,F)$ can be bounded with a function that depends on the so called \emph{logarithmic capacity} of the condenser with plates $E$ and $F$ \cite{Gonchar69}. In particular, Ganelius in \cite{Ganelius80} showed the inequality
\begin{equation}\label{eq:ganelius}
Z_h(E,F)\leq \gamma \rho^{-h},\qquad  \rho:=\exp\left(\frac{1}{\capac(E,F)}\right),
\end{equation}
where the constant $\gamma$ depends only on the geometry of $E$ and $F$ and $\capac(E,F)$ denotes the logarithmic capacity of the condenser with plates $E$ and $F$.
The decay rate in \eqref{eq:ganelius} is tight, i.e., $\lim_{h\to\infty}	Z_h(E,F)^{\frac 1h}=\rho^{-1}$, see \cite{Gonchar69}. However, the estimation of $\capac(E,F)$ is strongly problem dependent and its asymptotic behavior, when the plates approach each other, has been the subject of quite recent investigations, see \cite{Bonnafe2016} and the references therein.

A rather different approach, for getting singular values inequalities, has been suggested by Grasedyck et al.\cite{Grasedyck2004a}. Letting $\Gamma_A$, $\Gamma_B$ denote disjoint contours encircling the spectra of $A$ and $-B$, respectively, the solution $\delta X$ of~\eqref{eq:compactdelta} admits the integral representation
\[
\delta X=\frac{1}{4\pi^2}\int_{\Gamma_A}\int_{\Gamma_B}\frac 1{\xi-\eta}(\xi I-A)^{-1} D (\eta I +B)^{-1}d\xi d\eta.
\]
Then $\Gamma_B$ is split up into $s$ parts such that $\frac 1{\xi-\eta}$ admits a good semi-separable (polynomial) approximation on each part. Each approximation corresponds to a low-rank matrix and by summing up all $s$ parts of $\Gamma_B$ one obtains a low-rank approximation of $\delta X$. Although this technique is shown to establish exponential decay for certain shapes of $\Gamma_A$ and $\Gamma_B$, a small relative separation may require to use a very large $s$, resulting in unfavorable decay rates.

%% file: section3.tex
\section{Updating algorithm and application to algebraic Riccati equations}

\label{sec:alg-and-ric}

Algorithm~\ref{alg:split} summarizes the procedure outlined in the introduction for updating the solution $X_0$ of $A_0X_0 + X_0 B_0 = C_0$ such that $X_0 + \delta X$ approximates the solution of the perturbed Sylvester equation~\eqref{eq:update-sylv}. In the following, we discuss details of the implementation of Algorithm~\ref{alg:split} and then present a modification for Lyapunov equations as well as an application to Riccati equations.

\begin{algorithm}
	\caption{ Strategy for solving $(A_0+\delta A)(X_0+\delta X)+(X_0+\delta X)(B_0+\delta B)=C_0+\delta C$}\label{alg:split}
	\begin{algorithmic}[1]
		\Statex{{\bf procedure}} update\_Sylv($A_0,\delta A, B_0,\delta B,C_0,\delta C,X_0$)\Comment{$\delta A,\delta B$ and $\delta C$ have low rank}
		\State $\quad\!$ \label{step:rhs} Compute $U,V$ such that $\delta C-\delta A X_0-X_0\delta B = UV^*$
		\State $\quad\!$ \label{step:lrsolver} $\delta X\gets$ \Call{low\_rank\_Sylv}{$A_0+\delta A,B_0+\delta B,U,V$}
		\State $\quad$ \Return $X_0+\delta X$
		\Statex{{\bf end procedure}} 
	\end{algorithmic}
\end{algorithm}

\subsection{Step~\ref{step:rhs}: Construction of low-rank right-hand side}\label{sec:compress}

Given (low-rank) factorizations of the perturbations to the coefficients,
\[
\delta A=U_AV_A^*,\quad \delta B=U_BV_B^*,\quad \delta C=U_CV_C^*,
\]
a factorization $\delta C-\delta A X_0-X_0\delta B = UV^*$ can be cheaply obtained by simply setting
\begin{equation} \label{eq:uncompressedfactors}
U=[U_C, -U_A, -X_0U_B],\qquad V=[V_C, X_0^*V_A, V_B],
\end{equation}
where $U,V$ both have $s = \text{rank}(\delta A)+\text{rank}(\delta B)+\text{rank}(\delta C)$ columns.

The computational cost of low-rank solvers for Sylvester equations, such as the extended Krylov subspace method discussed in the next section, critically depends on the rank of the right-hand side. It is therefore advisable to perform a compression of the factors~\eqref{eq:uncompressedfactors} in order to decrease the rank. For this purpose, we compute 
reduced QR decompositions $U = Q_UR_U$ and $V = Q_VR_V$ such that $Q_U \in \mathbb R^{m\times s}$, $Q_V \in \mathbb R^{n\times s}$ have orthonormal columns and $R_U, R_V \in \mathbb R^{s\times s}$ are upper triangular. We then compute the singular values $\sigma_1,\ldots,\sigma_s$ of the $s\times s$ matrix
$R_UR_V^*$. We only retain the first $\tilde s \le s$ singular values, such that $\sigma_{\tilde s+1} / \sigma_1 \le \tau_{\sigma}$ for a user-specified tolerance $\tau_{\sigma} > 0$.
Letting $U_R$ and $V_R$ contain the first $\tilde s$ left and right singular vectors, respectively, and $\Sigma := \text{diag}(\sigma_1,\ldots,\sigma_{\tilde s})$, we set
\[
\widetilde U:=Q_UU_R\sqrt{\Sigma},\qquad \widetilde
V:=Q_VV_R\sqrt{\Sigma}.
\]
By construction, $\frac{\norm{\widetilde U\widetilde V^*-UV^*}_2}{\norm{UV^*}_2}\leq
\tau_{\sigma}$ and we can therefore safely replace the factorization $UV^*$ by $\widetilde U\widetilde V^*$.
Dominated by the computation of the two QR decompositions and the SVD, the described compression procedure requires $\mathcal O((m+n)s^2+s^3)$. In our setting, this method is attractive because $s\ll \min\{n,m\}$.


\subsection{Step~\ref{step:lrsolver}: Solution of correction equation}\label{sec:lrsolver}
Step~\ref{step:lrsolver} of Algorithm~\ref{alg:split} requires to solve a Sylvester equation of the form
\begin{equation}\label{eq:lr-sylv}
AX+XB=UV^*,
\end{equation} 
where $A = A_0+\delta A$, $B = B_0+\delta B$, and  $U,V$ have $s\ll \min\{n,m\}$ columns. We assume that $X$ can be well approximated by a low-rank matrix; which is the case -- for example -- when the 
hypotheses of
Theorem~\ref{thm:beckermann} are satisfied with two sets $E$ and $F$ 
ensuring a fast decay of $Z_\ell(E,F)$. The most common solvers for~\eqref{eq:lr-sylv} are ADI-type methods and
Krylov subspace projection algorithms~\cite{Simoncini2016}. One particularly effective approach to obtain a low-rank approximation of $X$ is the rational Krylov subspace method~\cite{Benner2013b} with the poles determined, e.g., by the rational approximation from Theorem~\ref{thm:beckermann}. On the other hand, the \emph{extended
	Krylov subspace method} introduced in \cite{Simoncini2007} does not require the estimation of such parameters and has been observed to be quite effective for many situations of interest. In the following, we therefore use this method for its simplicity but stress that any of the low-rank solvers mentioned above can be used instead.

The extended Krylov subspace method constructs orthonormal bases $U_t$ and $V_t$ of the subspaces
\begin{align*}
\mathcal U_t:&=\mathcal {EK}_t(A,U)= \Span\{U,A^{-1}U,AU,A^{-2}U,\dots,A^{t-1}U,A^{-t}U\},\\
\mathcal V_t:&=\mathcal {EK}_t(B^*,V)=\Span\{V,(B^*)^{-1}V,B^*V,(B^*)^{-2}V,\dots,(B^*)^{t-1}V,(B^*)^{-t}V\},
\end{align*}
for some $t$ satisfying $2ts<\min\{m,n\}$. This is done by means of
two extended block Arnoldi processes. Then,
the compressed equation
\[
\widetilde A X_t+X_t\widetilde B= \tilde C,\quad \widetilde A=
U_t^*AU_t,\quad \widetilde B= V_t^*BV_t,\quad \widetilde C=
U_t^*UV^*V_t,
\]
is solved by the Bartels-Stewart method~\cite{Bartels1972}. Note that the matrices $\widetilde A,\widetilde B$ and $\widetilde C$ do not
need to be computed explicitly, but can be obtained from the coefficients generated during the extended block Arnoldi processes; see~\cite[Proposition 3.2]{Simoncini2007} for more details. 
The matrix $\widetilde X = U_tX_tV_t^*$ is returned as approximation to the
solution of~\eqref{eq:lr-sylv}. The resulting method is summarized in Algorithm~\ref{alg:kryl}.  

\begin{algorithm}
	\caption{Extended Krylov subspace method for solving $AX+XB=UV^*$}\label{alg:kryl}
	\begin{algorithmic}[1]
		\Procedure{low\_rank\_Sylv}{$A,B,U,V$}\Comment{$U,V$ have both $s$ columns}
		\State $[U_1,-]=\texttt{qr}([U,A^{-1}U])$,\quad $[V_1,-]=\texttt{qr}([V,A^{-1}V])$
		\For{$t=1,2,\dots$}
		
		\State $\widetilde A\gets U_t^*A U_t$,\quad $\widetilde B\gets V_t^*BV_t$,\quad $\widetilde C\gets U_t^*UV^*V_t$
		\State $X_t\gets$  \Call{dense\_Sylv}{$\widetilde A,\widetilde B,\widetilde C$}
		\If{Converged}
		\State \Return $\widetilde X =U_tX_tV_t^*$
		\EndIf
		\State Select the last $2s$ columns: $U_t=[U^{(0)},U^{(+)},U^{(-)}]$,\quad $V_t=[V^{(0)},V^{(+)},V^{(-)}]$
		\State $\widetilde U=[AU^{(+)},A^{-1}U^{(-)}]$,\quad  $\widetilde V=[B^*V^{(+)},(B^*)^{-1}V^{(-)}]$
		\State $\widetilde U\gets \widetilde U-U_tU_t^*\widetilde U,$ \quad $\widetilde V\gets\widetilde V-V_tV_t^*\widetilde V$
		\Comment{Orthogonalize w.r.t $U_t$ and $V_t$}
		\State $[\tilde U,-]=\texttt{qr}(\tilde U),$\quad $[\tilde V,-]=\texttt{qr}(\tilde V)$
		\State $U_{t+1}=[U_t,\tilde U]$,\quad $V_{t+1}=[V_t,\tilde V]$
		\EndFor 
		\EndProcedure
	\end{algorithmic}
\end{algorithm}

A few remarks concerning the implementation of Algorithm~\ref{alg:kryl}:
\begin{itemize}
	\item 
	The number of iterations $t$ is chosen
	adaptively 
    to ensure that the relation
	\begin{equation}\label{eq:stop}
	\norm{A\widetilde X +\widetilde X B-UV^*}_2 \leq \tau
	\end{equation} is satisfied for some tolerance $\tau$, 
	which is chosen to be small 
	relative to the norm of $\widetilde X$, 
	i.e., $\tau = \tau_{\mathsf{EK}}
	\cdot \norm{\widetilde X}_2$ for a small $\tau_{\mathsf{EK}}$. This relation can be efficiently verified as described in \cite{Simoncini2007,Heyouni2008}. 
	
	\item The matrices generated in line $11$ of
	Algorithm~\ref{alg:kryl} might become numerically rank deficient. Several techniques have been proposed
	to deal with this phenomenon, see \cite{Birk,Gutknecht2006} and the references therein. We use the strategy described in \cite[Algorithm 7.3]{Birk},
	which performs pivoted QR decompositions in
	line $13$ and only retains columns corresponding to nonnegligible diagonal
	entries in the triangular factors.
	This reduces the size of the block vectors in all subsequent steps.
	
	\item For applying $A^{-1}$, $B^{-1}$ in Algorithm~\ref{alg:kryl}, (sparse) 
	LU factorizations of $A$ and $B$ are computed once before starting the extended block Arnoldi process. 
	
	\item When the algorithm is completed, we perform another compression of the returned solution by computing the truncated SVD of $X_t$ and using the same threshold $\tau_\sigma$ employed for compressing the right hand side. 
\end{itemize}

\subsection{Residual}\label{sec:residual}



As the correction equation is solved iteratively and inexactly, until the stopping criterion~\eqref{eq:stop} is satisfied, it is important to relate the accuracy of the solution 
$X$ to the accuracy of $X_0$ and $\delta X$. Suppose that the computed approximations $\hat X_0$ and 
$\delta \hat X$ satisfy
\[
  \norm{A_0 \hat X_0 + \hat X_0 B_0 - C_0} \leq \tau_0, \qquad  \norm{(A_0 + \delta A) \delta \hat X_0 + \delta \hat X_0 (B_0 + \delta B) - (\delta C - \delta A \hat X_0 - \hat X_0\delta B)} \leq \tau_\delta.
\] 
By the triangular
inequality, we can then conclude that $\hat X = \hat X_0 + \delta \hat X$ satisfies
\[
  \norm{(A_0 + \delta A) \hat X + \hat X (B_0 + \delta B) - (C_0 + \delta C)} \leq \tau_0 + \tau_\delta. 
\]
Hence, in order to avoid unnecessary work, it is advisable to choose $\tau_\delta$ not smaller than $\tau_0$.

\subsection{Stable Lyapunov equations}\label{sec:stable}

We now consider the special case of a Lyapunov equation $A_0 X_0 + X_0 A_0^* = C_0$, where $A_0 \in \mathbb C^{n\times n}$ is stable and $C_0 \in \mathbb C^{n\times n}$ is Hermitian negative semidefinite. We assume that the perturbed equation $A (X_0+\delta X) + (X_0+\delta X) A^* = C_0+\delta C$, with $A = A_0+\delta A$, has the same properties, implying that both $X_0$ and $X_0+\delta X$ are Hermitian positive semidefinite. In general, this does \emph{not} ensure that the corresponding correction equation 
\begin{equation}\label{eq:stable-lyap}
A\, \delta X+ \delta X\, A^* = \delta C - \delta A\, X_0 -X_0\, \delta A^*
\end{equation}
inherits these properties. In particular, the right-hand side may be indefinite. In turn, large-scale solvers tailored to stable Lyapunov equations with low-rank right-hand side~\cite{Li2004} cannot be directly applied to~\eqref{eq:stable-lyap}. Following~\cite[Sec. 2.3.1]{Benner2011a}, this issue can be addressed by splitting the right hand side. 
 
To explain the idea of splitting, suppose we have low-rank factorizations
$\delta A=U_AV_A^*$, $\delta C=U_C\Sigma_CU_C^*$, with $\Sigma_C$ Hermitian. In turn, the right-hand side of~\eqref{eq:stable-lyap} can be written as
\[
 \delta C - \delta A X_0 -X_0 \delta A^*=\widetilde U\Sigma\widetilde U^* \quad \text{with} \quad  \widetilde U = [U_C,U_A,X_0V_A],\quad \Sigma=\begin{bmatrix}
 \Sigma_C\\
 &&-I\\
 &-I
 \end{bmatrix}.
\]
After computing a reduced QR factorization $\widetilde U = Q_{\widetilde U}R_{\widetilde U}$, we compute a (reduced) spectral decomposition
of $R_{\widetilde U}\Sigma R_{\widetilde U}^*$ such that 
\[
R_{\widetilde U}\Sigma R_{\widetilde U}^* = \begin{bmatrix}
                                 Q_1 & Q_2
                                \end{bmatrix}
\begin{bmatrix}
                                 D_1 & 0 \\ 0 & -D_2
                                \end{bmatrix}
\begin{bmatrix}
                                 Q_1 & Q_2
                                \end{bmatrix}^*,
\]
where $D_1,D_2$ are diagonal matrices with positive diagonal elements. After setting $U_1=Q_{\widetilde U}Q_1\sqrt{D_1}$, $U_2=Q_{\widetilde U}Q_2\sqrt{D_2}$, this allows us to write 
$R_{\widetilde U}\Sigma R_U^* = U_1U_1^*-U_2U_2^*$. Hence, after solving the two stable Lyapunov equations
 \begin{equation} \label{eq:twostable}
A\,\delta X_1 + \delta X_1\, A^* = -U_1U_1^*,\qquad A\,\delta X_2 + \delta X_2\, A^* = -U_2U_2^*,
\end{equation}
the solution of \eqref{eq:stable-lyap} is obtained as $\delta X=\delta X_2-\delta X_1$.

The extended Krylov subspace method applied to~\eqref{eq:twostable} operates with the subspaces $\mathcal {EK}_t(A,U_1)$ and $\mathcal {EK}_t(A,U_2)$. This is more favorable than a naive application of the method to the original non-symmetric factorization $[U_C,-U_A,-X_0V_A][U_C,X_0V_A,U_A]^*$, which would operate with two subspaces of double dimension.

Another approach, which does not require the splitting of the
right hand side, relies on projecting the Lyapunov equation
onto the extended Krylov subspace $\mathcal {EK}(A, \widetilde U)$. In this way, we only need to generate a single Krylov subspace, even though with a larger block vector. In our
experience, the performances of the two approaches
are analogous. 
 
\subsection{Solving algebraic Riccati equation by the Newton method}\label{sec:riccati}
We now present a practical situation that requires to solve several perturbed Lyapunov equations sequentially. 

 Consider the \emph{continuous-time algebraic Riccati equation} (CARE)
\begin{equation}\label{eq:riccati}
AX+XA^*-XBX=C
\end{equation}
where $A\in\mathbb C^{n\times n}$ is a general matrix, while $B\in\mathbb C^{n\times n}$ is Hermitian positive semidefinite and $C\in\mathbb C^{n\times n}$ is Hermitian negative semidefinite. We also assume that the pair $(A^*,B)$ is \emph{stabilizable}, i.e., there exists $X_0\in\mathbb C^{n\times n}$ such that $A-X_0B$ is stable. Moreover, we suppose that $(C, A^*)$ is detectable, that is equivalent to the stabilizability
of $(A, C^*)$. 
Under these assumptions, there exists a unique Hermitian positive semi-definite solution $X\in\mathbb C^{n\times n}$ of \eqref{eq:riccati} such that $A-XB$ is stable~\cite{Kuvcera1992}. This is called the \emph{stabilizing} solution. 

Kleinman's formulation~\cite{Kleinman1968} of the Newton method~\eqref{eq:riccati} requires to solve the matrix equation
\begin{equation}\label{eq:newt-it}
 (A-X_kB)X_{k+1}+X_{k+1}(A^*-BX_k)=C-X_kBX_k,
 \end{equation}
 for determining the next iterate $X_{k+1}$. Assuming that the starting matrix $X_0$ is Hermitian, the matrices $X_k$ are Hermitian too and \eqref{eq:newt-it} is a Lyapunov equation with Hermitian right-hand side.
 
Under mild hypotheses, any Hermitian starting matrix $X_0$ such that $A-X_0B$ is stable, yields a quadratically convergent Hermitian sequence whose limit is the stabilizing solution \cite{Lancaster1995}. Moreover, the sequence is non-increasing in terms of the Loewner ordering. If $A$ is already stable, a natural choice is $X_0=0$.

In many examples of linear-quadratic control problems~\cite{Abels1999}, the coefficient $B$ has low-rank, i.e., it takes the form $B=B_UB_U^*$ where $B_U$ only has a few columns. In turn, two consecutive equations~\eqref{eq:newt-it} can be linked via low-rank updates. More explicitly,~\eqref{eq:newt-it} can be rewritten as
\[
(A_{k-1} + (X_{k-1}-X_k)B)X_{k+1} +X_{k+1}(A_{k-1}^*+B(X_{k-1}-X_k))= C_{k-1}+X_{k-1}BX_{k-1}-X_kBX_k, 
\]
where $A_{k-1}:= A-X_{k-1}B$ and $C_{k-1}:=C-X_{k-1}BX_{k-1}$.

Thus, after the --- possibly expensive --- computation of $X_1$ one can compute the updates $\delta X_k:=X_{k+1}-X_k$ by solving
\begin{equation}\label{eq:corr-lyap}
 (A-X_kB)\delta X_k+\delta X_k(A^*-BX_k)=\delta X_{k-1}B\delta X_{k-1},\quad k = 1,2,\dots.
\end{equation}
For this purpose, we use a variant of Algorithm~\ref{alg:kryl} tailored to Lyapunov equations, denoted by \textsc{low\_rank\_Lyap}. In contrast to the more general situation discussed in Section~\ref{sec:stable}, the right-hand side is always positive semi-definite and therefore no splitting is required.
The resulting method is described in Algorithm~\ref{alg:riccati}. For large scale problems, matrix $A_k$ at line $7$ is not formed explicitly and we rely on the Sherman-Morrison-Woodbury formula to compute the action of $A_k^{-1}$. If the final solution is rank structured then re-compression is advised after the sum at line $10$.
  \begin{algorithm}
  	\caption{Low-rank update Newton method for solving $AX+XA^*-XB_UB_U^*X=C$}\label{alg:riccati}
  	\begin{algorithmic}[1]
  		\Procedure{newt\_Riccati}{$A,B_U,C,X_0$}
  		\State $A_0\gets A-X_0B_UB_U^*$
  		\State $C_0\gets C- X_0B_UB_U^*X_0$	  	
  		\State \label{line:firstlyapnewton} $X_1\gets $\Call{solve\_Lyap}{$A_0,C_0$}\Comment{Any Lyapunov solver}
  		\State $\delta X_0\gets X_1-X_0$
  		\For{$k=1,2,\dots$}
  		
  		\State $A_k\gets A-X_{k}B_UB_U^*$
  		\State $C_U\gets \delta X_{k-1}B_U$

  		\State $\delta X_k\gets$  \Call{low\_rank\_Lyap}{$ A_k,C_U$}\label{line:corr-lyap}
  		\State $X_{k+1} = X_k+\delta X_k$
  		\If{$\norm{\delta X_k}_2<\tau_{\mathsf{NW}}\cdot \norm{X_1}_2$} \label{line:stopnewton}
  		\State \Return $ X_{k+1}$
  		\EndIf
  		\EndFor 
  		\EndProcedure
  	\end{algorithmic}
  \end{algorithm}
  
\subsubsection{Numerical experiments}\label{sec:ric-test}

We demonstrate the performance of Algorithm~\ref{alg:riccati} with two examples. Details of the implementation, the choice of parameters, and the computational environment are discussed in Section~\ref{sec:details} below.

%

  \begin{example}\label{exa:1}
 We first consider the convective thermal flow problem from the benchmark collection~\cite{Korvink2005}, leading to a CARE with coefficients
 $A\in\mathbb{R}^{9669\times 9669}$, $B_U\in\mathbb R^{9669\times 1}$, $C=-C_UC_U^*$ for $C_U\in\mathbb R^{9669\times 5}$.  
 The matrix $A$ is symmetric negative definite and sparse, only $67391$ entries are nonzero.
When applying the standard Newton method to this CARE, the right-hand side of the Lyapunov equation~\eqref{eq:newt-it}, that needs to be solved in every step, has rank at most $6$. On the other hand, the right-hand side of equation~\eqref{eq:corr-lyap} has rank $1$. As the computational effort of low-rank solvers for Lyapunov equations typically grows linearly with the rank, this makes Algorithm~\ref{alg:riccati} more attractive. In particular, this is the case for the extended Krylov subspace method, Algorithm~\ref{alg:kryl}, used in this work.

The performance of both variants of the Newton method is reported in Table~\ref{tab:exp1}. While $T_{\text{tot}}$ denotes the total execution time (in seconds), $T_{\text{avg}}$ denotes the average time needed for solving the Lyapunov equation in every step of the standard Newton method or Algorithm~\ref{alg:riccati} (after the first step). The quantity $\frac{T_{\texttt{step 1}}}{T_{\text{tot}}}$ shows the fraction of time spent by Algorithm~\ref{alg:riccati} on the (more expensive) first step. 
\emph{it} refers to the number of iterations and \emph{Res} refers to the relative residual $\norm{A\hat X+\hat XA^*-\hat XB\hat X-C}_2 / {\norm{AX_0+X_0A^*-X_0BX_0-C}_2}$ of the approximate solution $\hat X$ returned by each of the two variants.
The results reveal that Algorithm~\ref{alg:riccati} is faster while delivering the same level of accuracy.  	
  		\end{example}
  		
  		\begin{table}[t]
  			\centering 		
  			\small
  			\begin{filecontents*}{e3.dat}
9669	53.346	7.5868	0.14222	4.1599	1.2063e-07	12	89.839	7.4865	1.6502e-07	13	29.573
\end{filecontents*}
\pgfplotstabletypeset[%
  			column type=l,
  			every head row/.style={
  				before row={
  					\toprule
  					& \multicolumn{7}{c|}{ ~Algorithm~\ref{alg:riccati}~~~~~~~~} & \multicolumn{4}{ c}{ Standard Newton method}\\
  				},
  				after row = \midrule ,
  			}, 	
  			every last row/.style={after row=\bottomrule},		
  			sci zerofill,
  			columns={0,11,1,3,4,5,6,7,8,9,10},
  			columns/0/.style={column name=$n$},
  			columns/11/.style={column name=$\norm{\hat X}_2$,sci},
  			columns/1/.style={column name=$T_{\text{tot}}$},
  			columns/3/.style={column name=$\frac{T_{\texttt{step 1}}}{T_{\text{tot}}}$},
  			columns/4/.style={column name=$T_{\text{avg}}$},    
  			columns/5/.style={column name=Res},
  			columns/6/.style={dec sep align={c|},column type/.add={}{|},column name=it},
  			columns/7/.style={column name=$T_{\text{tot}}$},
  			columns/8/.style={column name=$T_{\text{avg}}$},
  			columns/9/.style={column name=Res},
  			columns/10/.style={column name=it}
  			]{e3.dat} 
  			\caption{Performance of Algorithm~\ref{alg:riccati} and the standard Newton method for the CARE from Example~\ref{exa:1}.}
  			\label{tab:exp1}
  		\end{table}
  	
  	  	\begin{table}[h]
  	  		\centering 		
  	  		\small
  	  		\begin{filecontents*}{e2.dat}
512	1.211	0.47552	0.39267	0.073547	3.4237e-09	11	5.3743	0.48857	1.6911e-13	12	15489
1024	6.2317	5.2346	0.84	0.090645	1.5196e-08	12	53.184	4.432	4.7009e-13	13	61942
1536	22.7	20.722	0.91285	0.17985	3.045e-08	12	252.58	21.048	7.9476e-13	13	1.3936e+05
2048	62.371	59.591	0.95543	0.25269	5.7824e-08	12	735.99	61.333	1.3403e-12	13	2.4776e+05
\end{filecontents*}
\pgfplotstabletypeset[%
  	  		column type=l,
  	  		every head row/.style={
  	  			before row={
  	  				\toprule
  	  				& \multicolumn{7}{c|}{Algorithm~\ref{alg:riccati}~~~~~~~~} & \multicolumn{4}{c}{\quad Standard Newton method}\\
  	  			},
  	  			after row = \midrule,
  	  		}, 	
  	  		every last row/.style={after row=\bottomrule},		
  	  		sci zerofill,
  	  		columns={0,11,1,3,4,5,6,7,8,9,10},
  	  		columns/0/.style={column name=$n$},
  	  		columns/11/.style={column name=$\norm{\hat X}_2$,sci},
  	  		columns/1/.style={column name=$T_{\text{tot}}$},
  	  		columns/3/.style={column name=$\frac{T_{\texttt{step 1}}}{T_{\text{tot}}}$}, 
  	  		columns/4/.style={column name=$T_{\text{avg}}$, fixed},   
  	  		columns/5/.style={column name=Res},
  	  		columns/6/.style={dec sep align={c|},column type/.add={}{|},column name=it},
  	  		columns/7/.style={column name=$T_{\text{tot}}$},
  	  		columns/8/.style={column name=$T_{\text{avg}}$},
  	  		columns/9/.style={column name=Res},
  	  		columns/10/.style={column name=it}
  	  		]{e2.dat}
  	  		\caption{Performance of Algorithm~\ref{alg:riccati} and the standard Newton method for the CARE from Example~\ref{exa:2}.}
  	  		\label{tab:exp2}
  	  	\end{table}
  	  		
  	\begin{example}\label{exa:2}
  		We now consider Example 4.3 from~\cite{Abels1999}, a CARE with coefficients
  		\begin{align*}
  		A&=\begin{bmatrix}
  		0&-\frac{1}{4}K\\
  		I_q&-I_q
  		\end{bmatrix},\quad K={\small \begin{bmatrix}
  		1&-1\\
  		-1&2&-1\\
  		&\ddots&\ddots&\ddots\\
  		&&-1&2&-1\\
  		&&&-1&1
  		\end{bmatrix}} \in\mathbb R^{q \times q},\\	
  		B&=B_UB_U^*,\qquad B_U=\begin{bmatrix}
  		0\\ \frac 14 D
  		\end{bmatrix}\in\mathbb R^{2q\times 2},\qquad  D=\begin{bmatrix} e_1 & e_q 
  		\end{bmatrix}\in\mathbb R^{q\times 2},\quad C = I_{2q},
  		\end{align*}
  		where $e_j$ denotes the $j$th unit vector of length $q$.
  		Except for one zero eigenvalue, the spectrum of $A$ is contained in the open left half
                plane. A stabilizing initial guess is given by 
  		\[
  		X_0 =  
  		EE^*, \quad E=2\begin{bmatrix}
  		-e_q & e_1\\
  		-e_q&e_1
  		\end{bmatrix}\in\mathbb R^{2q\times 2}.
  		\] 
  		Note that $C$ has full rank, which prevents us from the use of low-rank solvers for addressing the Lyapunov equation~\eqref{eq:newt-it} in the standard Newton iteration and we need to resort to the (dense) Bartel-Stewart method implemented in the MATLAB function \texttt{lyap}. In contrast, the right-hand side of~\eqref{eq:corr-lyap} has rank $2$, which allows us to use such low-rank solvers in every but the first step of Algorithm~\ref{alg:riccati}.  
  		
  		The obtained results, for varying $n:=2q$, are shown in Table~\ref{tab:exp2}. Not surprisingly,
                Algorithm~\ref{alg:riccati} is much faster in this example because it
                only relies on the dense solver in the first
                step. 
  	\end{example}


%% file: section4.tex
\section{A divide-and-conquer approach} \label{sec:dac}

In this section we use low-rank updates to devise a new divide-and-conquer method for the Sylvester equation~\eqref{eq:sylv}. For simplicity, we consider the case $m=n$ and hence the solution $X$ is a square $n\times n$ matrix. In principle, our developments extend to the case $m\neq n$ but additional technicalities come into play, e.g., the hierarchical low-rank formats need to be adjusted.

 Suppose that the coefficients of~\eqref{eq:sylv} can be decomposed as  
\begin{equation} \label{eq:decompcoeff}
A=A_0+\delta A,\qquad
B=B_0+\delta B,\qquad
C=C_0+\delta C,
\end{equation}
where $A_0,B_0,C_0$ are block diagonal matrices of the same shape and the corrections $\delta A, \delta B, \delta C$ all have low rank. This is, in particular, the case when all coefficients are banded but~\eqref{eq:decompcoeff} allows to handle more general situations. We apply Algorithm~\ref{alg:split} to deal with the low-rank corrections. It thus remains to solve the smaller Sylvester equations associated with the diagonal blocks of $A_0X_0+X_0B_0 = C_0$. If the diagonal blocks of $A_0,B_0,C_0$ again admit a decomposition of the form~\eqref{eq:decompcoeff}, we can recursively repeat the procedure. Keeping track of the low-rank corrections at the different levels of the recursions requires us to work with hierarchical low-rank formats, such as the hierachically off-diagonal low-rank (HODLR) and the hierarchically semi-separable (HSS) formats.

\subsection{HODLR matrices}
A HODLR matrix $A\in \mathbb C^{n\times n}$, as defined in \cite{Ballani2016,Ambikasaran2013}, admits a block partition of the form
\begin{equation} \label{eq:hodler1level}
A = \begin{bmatrix}
	A_{11}&A_{12}\\
	A_{21}&A_{22}
	\end{bmatrix},
\end{equation}
where $A_{12}$, $A_{21}$ have low rank and $A_{11}$, $A_{22}$ are again of the form~\eqref{eq:hodler1level}. This recursive partition is continued until the diagonal blocks have reached a certain minimal block size. For later purposes, it is helpful to give a formal definition of the HODLR format that proceeds in the opposite direction, from the finest block level to the full matrix. Given an integer $p$, let us consider an integer partition
\begin{equation} \label{eq:partn}
 n = n_1 + n_2 + \cdots + n_{2^p},
\end{equation}
where $p$ and $n_i$ are usually chosen such that all $n_i$ are nearly equal to the minimal block size. We build a perfectly balanced binary tree, the so called cluster tree, from this partition by setting $n_i^{(p)} := \sum_{j=1}^i n_j$ and defining the leaf nodes to be
\[
 I^p_1 = \{1,\ldots,n^{(p)}_1\}, \quad I^p_2 = \{n^{(p)}_1+1,\ldots,n^{(p)}_2\}, \quad \ldots \quad I^p_{2^p} = \{n^{(p)}_{2^p-1}+1,\ldots,n\} .
\]
The nodes of depth $\ell<p$ are defined recursively by setting $I^\ell_i = I^{\ell+1}_{2i-1} \cup I^{\ell+1}_{2i}$ for $i = 1,\ldots,2^\ell$. At the root, $I^0_1 = I:=\{1,\ldots,n\}$.
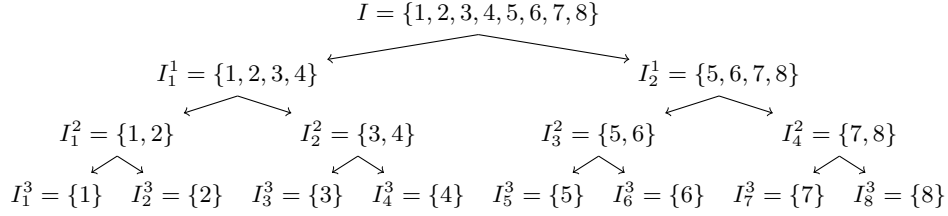
\begin{figure}
\centering 
 \begin{tikzpicture}[scale=0.8] \small
    \coordinate (T18) at (0,0);  
    \coordinate (T14) at (-4,-1);
    \coordinate (T58) at (4,-1);
    \coordinate (T12) at (-6,-2);
    \coordinate (T34) at (-2,-2);
    \coordinate (T56) at (2,-2);
    \coordinate (T78) at (6,-2);
    \coordinate (T1) at (-7,-3);
    \coordinate (T2) at (-5,-3);
    \coordinate (T3) at (-3,-3);
    \coordinate (T4) at (-1,-3);
    \coordinate (T5) at (1,-3);
    \coordinate (T6) at (3,-3);
    \coordinate (T7) at (5,-3);
    \coordinate (T8) at (7,-3);
    \node (N18) at (T18) {$I = \{1,2,3,4,5,6,7,8\}$};
    \node (N14) at (T14) {$I_1^1 = \{1,2,3,4\}$};
    \node (N58) at (T58) {$I_2^1 = \{5,6,7,8\}$};
    \node (N12) at (T12) {$I_1^2 = \{1,2\}$};
    \node (N34) at (T34) {$I_2^2 = \{3,4\}$};
    \node (N56) at (T56) {$I_3^2 = \{5,6\}$};
    \node (N78) at (T78) {$I_4^2 = \{7,8\}$};
    \node (N1) at (T1) {$I_1^3 = \{1\}$};
    \node (N2) at (T2) {$I_2^3 = \{2\}$};
    \node (N3) at (T3) {$I_3^3 = \{3\}$};
    \node (N4) at (T4) {$I_4^3 = \{4\}$};
    \node (N5) at (T5) {$I_5^3 = \{5\}$};
    \node (N6) at (T6) {$I_6^3 = \{6\}$};   
    \node (N7) at (T7) {$I_7^3 = \{7\}$};
    \node (N8) at (T8) {$I_8^3 = \{8\}$};
    \draw[->] (N18.south) -- (N14);
    \draw[->] (N18.south) -- (N58);  
    \draw[->] (N14.south) -- (N12);
    \draw[->] (N14.south) -- (N34);    
    \draw[->] (N58.south) -- (N56);  
    \draw[->] (N58.south) -- (N78);
    \draw[->] (N12.south) -- (N1);
    \draw[->] (N12.south) -- (N2);
    \draw[->] (N34.south) -- (N3);  
    \draw[->] (N34.south) -- (N4);  
    \draw[->] (N56.south) -- (N5);  
    \draw[->] (N56.south) -- (N6);          
    \draw[->] (N78.south) -- (N7);  
    \draw[->] (N78.south) -- (N8);          
\end{tikzpicture}
\caption{ Example of a cluster tree of depth $3$.}\label{fig:clustertree}
\end{figure}
Figure~\ref{fig:clustertree} provides an illustration of the cluster tree obtained for $n=8$, with the (impractical) minimal block size $1$. We use $\mathcal T_p$ to denote the cluster tree associated with~\eqref{eq:partn}.

\begin{figure}[!ht]
\centering
\begin{tikzpicture}[scale=0.3]    
      \begin{scope}
        [xshift=0cm]
        \node [above] at (4,8) {$\ell=0$};
        \draw (0,0) -- (0,8) -- (8,8) -- (8,0) -- cycle;
      \end{scope}
      \begin{scope}
        [xshift=10cm]
        \node [above] at (4,8) {$\ell=1$};
        \draw (0,0) -- (0,8) -- (8,8) -- (8,0) -- cycle;
        \draw[pattern=north west lines, pattern color=blue] (0,0) rectangle (4,4);
        \draw[pattern=north west lines, pattern color=blue] (4,4) rectangle (8,8);
        \draw (0,4) -- (8,4);
        \draw (4,0) -- (4,8);
      \end{scope}
      \begin{scope}
        [xshift=20cm]
        \node [above] at (4,8) {$\ell=2$};
        \draw (0,0) -- (0,8) -- (8,8) -- (8,0) -- cycle;
        \draw[pattern=north west lines, pattern color=blue] (4,0) rectangle (6,2);
        \draw[pattern=north west lines, pattern color=blue] (0,4) rectangle (2,6);
        \draw[pattern=north west lines, pattern color=blue] (6,2) rectangle (8,4);
        \draw[pattern=north west lines, pattern color=blue] (2,6) rectangle (4,8);
        \draw (0,2) -- (8,2);
        \draw (0,4) -- (8,4);
        \draw (0,6) -- (8,6);
        \draw (2,0) -- (2,8);
        \draw (4,0) -- (4,8);
        \draw (6,0) -- (6,8);
      \end{scope}
      \begin{scope}
        [xshift=30cm]
        \node [above] at (4,8) {$\ell=3$};
        \draw (0,0) -- (0,8) -- (8,8) -- (8,0) -- cycle;
        \draw[pattern=north west lines, pattern color=blue] (6,0) rectangle (7,1);
        \draw[pattern=north west lines, pattern color=blue] (4,2) rectangle (5,3);
        \draw[pattern=north west lines, pattern color=blue] (2,4) rectangle (3,5);
        \draw[pattern=north west lines, pattern color=blue] (0,6) rectangle (1,7);
        \draw[pattern=north west lines, pattern color=blue] (7,1) rectangle (8,2);
        \draw[pattern=north west lines, pattern color=blue] (5,3) rectangle (6,4);
        \draw[pattern=north west lines, pattern color=blue] (3,5) rectangle (4,6);
        \draw[pattern=north west lines, pattern color=blue] (1,7) rectangle (2,8);
        
        \draw (0,1) -- (8,1);
        \draw (0,2) -- (8,2);
        \draw (0,3) -- (8,3);
        \draw (0,4) -- (8,4);
        \draw (0,5) -- (8,5);
        \draw (0,6) -- (8,6);
        \draw (0,7) -- (8,7);
        \draw (1,0) -- (1,8);
        \draw (2,0) -- (2,8);
        \draw (3,0) -- (3,8);
        \draw (4,0) -- (4,8);
        \draw (5,0) -- (5,8);
        \draw (6,0) -- (6,8);
        \draw (7,0) -- (7,8);
      \end{scope}
\end{tikzpicture}

\caption{ Block partitions induced by each level of a cluster tree of depth $3$.}\label{fig:blockpart}
\end{figure}
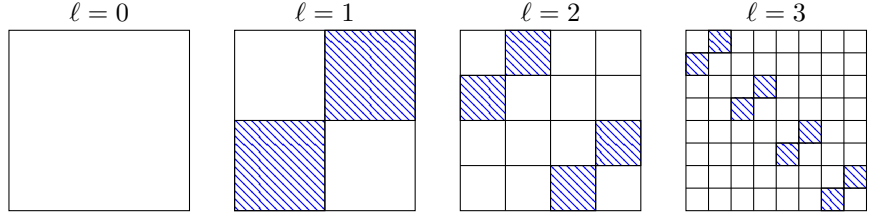

There are $2^\ell$ nodes on level $\ell$ of $\mathcal T_p$ and they partition a matrix $A\in \mathbb C^{n\times n}$ into a $2^\ell \times 2^\ell$ block matrix with the blocks $A(I_i^\ell, I_j^\ell)$ for $i,j = 1,\ldots,2^\ell$. For the HODLR format, only the off-diagonal blocks appearing in the recursive partition~\eqref{eq:hodler1level} are relevant. These are given by
\begin{equation} \label{eq:hodlrblocks}
 A(I_i^\ell,I_j^\ell) \quad \text{and} \quad A(I_j^\ell,I_i^\ell) \quad \text{for} \quad (i, j)=(2, 1), (4, 3), \dots, (2^\ell, 2^\ell-1), \quad \ell = 1,\ldots,p,
\end{equation}
and marked with stripes in Figure~\ref{fig:blockpart}.

\begin{definition} \label{def:hodlr}
Consider a cluster tree $\mathcal T_p$ and let $A\in \mathbb C^{n\times n}$. Then:
\begin{enumerate}
 \item for $k\in \mathbb N$, $A$ is said to be a $(\mathcal T_p,k)$-HODLR matrix if each of the off-diagonal blocks listed in~\eqref{eq:hodlrblocks} has rank at most $k$;
 \item the HODLR rank of $A$ (with respect to $\mathcal T_p$)  is the smallest integer $k$ such that $A$ is a $(\mathcal T_p,k)$-HODLR matrix.
\end{enumerate}
\end{definition}

A $(\mathcal T_p,k)$-HODLR matrix $A$ can be stored efficiently by replacing each off-diagonal block $A(I_i^\ell,I_j^\ell)$ featuring in Definition~\ref{def:hodlr} by its low-rank factorization $U_i^{(\ell)} \big(V_j^{(\ell)}\big)^*$. Both, $U_i^{(\ell)}$ and $V_j^{(\ell)}$ have at most $k$ columns.
In turn, the only full blocks that need to be stored are the diagonal blocks at the lowest level: $A(I_i^p,I_i^p)$ for $i = 1,\ldots,2^p$; see also Figure~\ref{fig:hodlr}.
If $n = 2^p k$ and $n_i = k$ in~\eqref{eq:partn}, $O(nk)$ memory is needed for storing the diagonal blocks as well as the low-rank factors on each level $\ell$. Hence, the storage of such a HODLR matrix requires $O(pnk) = O(kn\log n)$ memory in total.

\begin{figure}[!ht]
	\centering
	\includegraphics[width=0.8\textwidth]{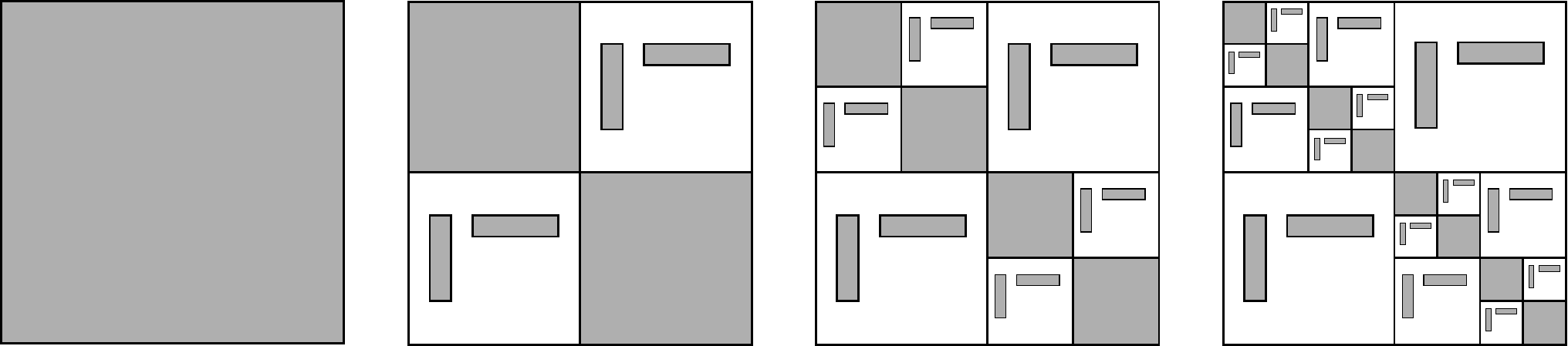}
	\caption{ Pictorial description of the HODLR format for cluster trees of different depths. The sub-blocks filled in gray are dense blocks; the others are stored as low-rank outer products.}\label{fig:hodlr}
\end{figure}

\subsection{Divide-and-conquer in the HODLR format}

We are now prepared to describe the divide-and-conquer method for a Sylvester equation~\eqref{eq:sylv} with ($\mathcal T_p,k$)-HODLR matrices $A$, $B$, $C$. By definition, we can write
\begin{equation} \label{eq:hodlrsplit}
 A = \begin{bmatrix}
      A_{11} & 0 \\
      0 & A_{22}
     \end{bmatrix}
+\begin{bmatrix}
      0 & A_{12} \\
      A_{21} & 0
     \end{bmatrix} = 
\begin{bmatrix}
      A_{11} & 0 \\
      0 & A_{22}
     \end{bmatrix}
+\begin{bmatrix}
      0 & U_1 V_2^*  \\
      U_2 V_1^* & 0
     \end{bmatrix} = 
\begin{bmatrix}
      A_{11} & 0 \\
      0 & A_{22}
     \end{bmatrix}
+U_A V_A^*, 
\end{equation}
where
\begin{equation} \label{eq:uava}
 U_A := \begin{bmatrix}
      U_1 & 0  \\
      0 & U_2 
     \end{bmatrix}, \qquad 
V_A := \begin{bmatrix}
      0 & V_1   \\
      V_2 & 0
     \end{bmatrix}
\end{equation}
have at most $2k$ columns each. The matrices $U_B$, $V_B$, $U_C$, $V_C$ are defined analogously. Note that the diagonal blocks are ($\mathcal T_{p-1},k$)-HODLR matrices for a suitably defined cluster tree $\mathcal T_{p-1}$. After solving (recursively) for these diagonal blocks, we apply the technique described in Algorithm~\ref{alg:split} to incorporate the low-rank corrections for $A$, $B$, $C$. The described procedure is summarized in Algorithm~\ref{alg:dac}.
\begin{algorithm}
	\caption{Divide-and-conquer method for Sylvester equations with HODLR coefficients}\label{alg:dac}
	\begin{algorithmic}[1]
		\Procedure{D\&C\_Sylv}{$A,B,C$}\Comment{Solve $AX+XB=C$}
		\If{$A,B$ are dense matrices}
		\State \Return \Call{dense\_Sylv}{$A,B,C$}
		\Else
		\State Decompose \[A =\begin{bmatrix}
		A_{11}&0 \\ 0&A_{22}
		\end{bmatrix}+U_AV_A^*,\ B=\begin{bmatrix}
		B_{11}&0\\ 0&B_{22}
		\end{bmatrix} + U_BV_B^*,\, C=\begin{bmatrix}
		C_{11}& 0\\ 0&C_{22}
		\end{bmatrix}+U_CV_C^*\]
		
		~~~with $U_A, V_A, U_B,V_B,U_C,V_C$ defined as in~\eqref{eq:uava}.
		\State $X_{11} \gets $ \Call{D\&C\_Sylv}{$A_{11},B_{11},C_{11}$}
			\State $X_{22} \gets $ \Call{D\&C\_Sylv}{$A_{22},B_{22},C_{22}$}
			\State \label{line:X0} Set $X_0\gets\begin{bmatrix}
			X_{11} & 0 \\ 0&X_{22}
			\end{bmatrix}$
		\State \label{line:largerlowrank} Set $U=[U_C, -U_A, -X_0U_B]$ and $V=[V_C, X_0^* V_A, V_B]$. 
		\State \label{line:lrupdate} $\delta X\gets $  \Call{low\_rank\_Sylv}{$A,B,U,V$}
		\State \label{step:hodlrcompress} \Return \Call{Compress}{$X_0+\delta X,\tau$}\Comment{Compression is optional}
		\EndIf
		\EndProcedure
	\end{algorithmic}
\end{algorithm}

When implementing Algorithm~\ref{alg:dac} it is advisable to recompress the right hand side $UV^*$, obtained in line~\ref{line:largerlowrank}, using the procedure described in Section~\ref{sec:compress}. 
Similarly, Line~\ref{step:hodlrcompress} aims at recompressing the entire HODLR matrix to mitigate the increase of the HODLR rank due to the addition of $\delta X$. For this purpose, the procedure from Section~\ref{sec:compress} is applied, with truncation tolerance $\tau$, to each off-diagonal block~\eqref{eq:hodlrblocks}. This step is optional because it is expensive and we cannot expect a significant rank  reduction for Sylvester equations with general HODLR coefficients; see also Remark~\ref{rem:rank-growth} below. 

When Algorithm~\ref{alg:dac} is applied to a stable Lyapunov equation, the techniques from Section~\ref{sec:stable} are employed in Line~\ref{line:lrupdate} in order to preserve the symmetry of $\delta X$. Note, however, that Algorithm~\ref{alg:dac} does not preserve definiteness, that is, $\delta X$ is, in general, not positive semidefinite. We pose it as an open problem to design a divide-and-conquer method that has this desirable property, implying that the solution is approached monotonically from below.

\subsubsection{A priori bounds on the HODLR rank of $X$}

The numerical rank of the correction $\delta X$, computed in line~\ref{line:lrupdate}, can be bounded using the tools introduced in Section~\ref{sec:low-rank}. 
\begin{lemma} \label{thm:uniform-approximability}
	Let $A,B,C\in\mathbb C^{n\times n}$ be $(\mathcal T_p,k)$-HODLR matrices and suppose that $\mathcal W(A)\subseteq E$, $\mathcal W(-B)\subseteq F$ for sets $E,F \subset \mathbb C$ satisfying $E\cap F=\emptyset$. and run Algorithm~\ref{alg:dac} with input arguments $A,B$ and $C$. Then for every recursion of Algorithm~\ref{alg:dac}, the correction $\delta X$ satisfies
	\begin{equation}\label{eq:uniform-approximability}
	\frac{\sigma_{6kh+1}(\delta X)}{\norm{\delta X}_2}\leq \big(1+\sqrt 2\big)\cdot Z_{h}(E,F).
	\end{equation}
\end{lemma}
\begin{proof}
As the matrices $A$ and $B$ appearing in each recursion of Algorithm~\ref{alg:dac} are principal submatrices of the input matrices $A$ and $B$, their numerical ranges are contained in $E$ and $-F$, respectively. Moreover, the rank of the right hand-side $UV^*$ is bounded by $6k$. Thus, applying  Theorem~\ref{thm:beckermann} to $A\,\delta X + \delta X\,B = UV^*$ establishes the claim.
\end{proof}

We now use Lemma~\ref{thm:uniform-approximability} to derive an apriori approximation result for $X$. Let $\delta X_\ell \in \mathbb C^{n\times n}$ be the block diagonal matrix for which the diagonal blocks contain all corrections computed at recursion level $\ell \le p-1$ of Algorithm~\ref{alg:dac}. Note that the block partitioning of $\delta X_\ell$ corresponds to level $\ell$ of $\mathcal T_p$; see also Figure~\ref{fig:blockpart}. Similarly, let $X_0\in \mathbb C^{n\times n}$ be the block diagonal matrix that contains all the solutions of dense Sylvester equations at level $p$. Then
$
 X = X_0 + \delta X_0 + \cdots + \delta X_{p-1}.
$
Given $\tilde \epsilon > 0$, suppose that $h$ is chosen such that $(1+\sqrt 2)Z_h(E,F)\leq \tilde \epsilon$. Lemma~\ref{thm:uniform-approximability} applied to each diagonal block of $\delta X_\ell$ implies that there is a block diagonal matrix $\delta \tilde X_j$, with the same block structure as $\delta X_\ell$, such that each diagonal block has rank at most $6kh$ and $\norm{\delta X_\ell -\delta\widetilde X_\ell}_2 \le \tilde \epsilon \|\delta X_\ell\|_2$. By construction,
$
 \tilde X = X_0 + \delta \tilde X_0 + \cdots + \delta \tilde X_{p-1}
$
is a $(\mathcal T_p,6kh p)$-HODLR matrix such that $\|X-\tilde X\|_2 \le \tilde \epsilon p \max_\ell \|\delta X_\ell\|_2$. This establishes the following result.
\begin{corollary}\label{cor:rank-increase}
	Under the assumptions of Lemma~\ref{thm:uniform-approximability}, let $X$ be the solution of~\eqref{eq:sylv} and suppose that the norm of all corrections computed by Algorithm~\ref{alg:dac} is bounded by $M$. Given $\epsilon>0$, let $h$ be the smallest integer that satisfies $(1+\sqrt 2)Z_h(E,F)\leq \frac{\epsilon }{pM}$. Then there exists a $(\mathcal T_p,6kh p)$-HODLR matrix $\widetilde X$  such that
	$
	\norm{X-\widetilde X}_2\leq \epsilon.
	$
\end{corollary}

\begin{remark}\label{rem:rank-growth}
To turn Corollary~\ref{cor:rank-increase} into an asymptotic statement on the HODLR rank as $n\to \infty$, one needs to assume a model for the behavior of $E$, $F$ as $n\to \infty$. In the simplest case, $E,F$ stay constant; for example, when $A$ and $B$ are symmetric positive definite and their condition numbers remain bounded as $n\to \infty$. In this case, the integer $h$ from Corollary~\eqref{cor:rank-increase} is constant and, in turn, the HODLR rank of the approximate solution is $\mathcal O(k\log(n))$. Numerical tests, involving random HODLR matrices which satisfy these assumptions, indicate that the factor $\log(n)$ is \emph{in general} not avoidable. 

In many cases of practical interest, $A$ and $B$ are symmetric positive definite but their condition numbers grow polynomially with respect to $n$. For example, the condition number of $A=B=\trid(-1,2,-1)$ is $\mathcal O(n^2)$. Using the result of  Corollary~\ref{cor:decay} one now has $h = \mathcal O(\log(n))$ and, in turn,  Corollary~\ref{cor:rank-increase} yields the HODLR rank $\mathcal O(k\log^2(n))$.

  	\end{remark}
\subsubsection{Complexity of divide-and-conquer in the HODLR format}\label{sec:complexity}

The complexity of Algorithm~\ref{alg:dac} depends on the convergence of the extended Krylov subspace method used for solving the correction equation in Line~\ref{line:lrupdate} and, in turn, on spectral properties of $A,B$; see~\cite{Beckermann2011,Knizhnerman2011}. To still give some insights into the complexity, we make the following simplifying assumptions:
  \begin{enumerate}[label=(\roman*)]
  	\item the conditions of Lemma~\ref{thm:uniform-approximability} are satisfied for sets $E,F$ independent of $n$, and Algorithm~\ref{alg:kryl} converges in a constant number of iterations;
	\item[(ii)] $n=2^ps$ and the input matrices $A,B$ and $C$ are $(\mathcal T_p,k)$-HODLR matrices;
	\item[(iii)] $\mathcal T_p$ is the perfectly balanced binary tree of depth $p$,
	\item[(iv)] the compression in Line~\ref{step:hodlrcompress} of Algorithm~\ref{alg:dac}  is \emph{not} performed.
\end{enumerate}
We recall that the LU decomposition of a $(\mathcal T_p,k)$-HODLR matrix requires $\mathcal O(k^2n\log^2(n))$ operations, while multiplying a vector with such a matrix requires $\mathcal O(kn\log(n))$ operations; see, e.g.,~\cite{Hackbusch2015}.

Now, let $\mathcal C(n,k)$ denote the complexity of Algorithm~\ref{alg:dac}. 
Assumption (i) implies that the cost of Algorithm~\ref{line:lrupdate}, called at  Line~\ref{line:lrupdate} is $\mathcal O(k^2n\log^2(n))$, because it is dominated by the cost of precomputing the LU decompositions for $A$ and $B$. According to Corollary~\ref{cor:rank-increase} and Remark~\ref{rem:rank-growth}, Assumption (i) also implies that $X_0$, see Line~\ref{line:X0}, has HODLR rank $\mathcal O(k\log(n))$. This, together with the fact that $U_B$ and $V_A$ each have $2k$ columns, shows that the matrix multiplications $X_0U_B$ and $X_0^*V_A$ at Line~\ref{line:largerlowrank} require $\mathcal O(k^2n\log^2(n))$ operations. 
Finally, observe that the solution of a dense Sylvester equation with $s\times s$ coefficients requires $\mathcal O(s^3)$ operations. In summary, 
\[
\mathcal C(n,k)= \begin{cases}
\mathcal O(s^3)&\text{if }n= s,\\
\mathcal O(k^2n\log^2(n)) + 2\mathcal C(\frac n2,k)&\text{otherwise}.
\end{cases}
\]
Applying the master theorem \cite[Theorem 4.1]{cormen} to this  relation yields  $\mathcal C(n,k)=\mathcal O(k^2n\log^3(n))$. 

\subsection{HSS matrices}

The storage of a matrix of HODLR rank $k$ requires $\mathcal O(k n\log n)$ memory in the HODLR format. Under stronger conditions on the matrix, the factor $\log n$ can be avoided by using nested hierarchical low-rank formats, such as the HSS format.

An HSS matrix is partitioned in the same way as a HODLR matrix. By Definition~\ref{def:hodlr}, 
a matrix $A$ is a $(\mathcal T_p,k)$-HODLR matrix if and only if every off-diagonal block $A(I^\ell_i,I^\ell_{j})$, $i\not=j$, has rank at most $k$. Thus, we have a factorization
\[
   A(I^\ell_i,I^\ell_j) = U_i^{(\ell)} S_{i,j}^{(\ell)} (V_j^{(\ell)})^*, \quad S_{i,j}^{(\ell)} \in\mathbb C^{k\times k}, \quad U_i^{(\ell)}\in \mathbb C^{n_i^{(\ell)} \times k},\quad V_j^{(\ell)} \in \mathbb C^{n_j^{(\ell)} \times k},
\]
where we assume exact rank $k$ to simplify the description.
The crucial extra assumption for HSS matrices is that the bases matrices of these low-rank representations are nested across the different levels. That is, one assumes that there exist matrices, the so called \emph{translation operators}, $R_{U,i}^{(\ell)}, R_{V,j}^{(\ell)}\in\mathbb C^{2k\times k}$ such that
\begin{equation}\label{eq:HSS_nestedness}
    U_i^{(\ell)} = \begin{bmatrix} U_{2i-1}^{(\ell+1)} & 0 \\ 0 & U_{2i}^{(\ell+1)} \end{bmatrix} R_{U,i}^{(\ell)}, \qquad
    V_j^{(\ell)} = \begin{bmatrix} V_{2j-1}^{(\ell+1)} & 0 \\ 0 & V_{2j}^{(\ell+1)} \end{bmatrix} R_{V,j}^{(\ell)}. 
\end{equation}
This nestedness condition allows to construct the bases $U_i^{(\ell)}$ and $V_i^{(\ell)}$ for any level $\ell=1,\ldots,p-1$ recursively from the bases $U_i^{(p)}$ and $V_i^{(p)}$ at the deepest level $p$ using the matrices $R_{U,i}^{(\ell)}$ and $R_{V,j}^{(\ell)}$. In turn, one only needs $O(nk)$ memory to represent $A$, for storing the diagonal blocks $A(I^p_i,I^p_i)$, the bases $U_i^{(p)}$, $V_i^{(p)}$ as well as $S_{2i-1,2i}^{(\ell)}$, $S_{2i,2i-1}^{(\ell)}$, $R_{U,i}^{(\ell)}$, $R_{V,i}^{(\ell)}$.

As explained in~\cite{Xia2010}, the described construction is possible if and only if all the corresponding block rows and columns, without their diagonal blocks, have rank at most $k$ on every level. The following definition formalizes and extends this condition.
\begin{definition} \label{def:hss}
Consider a cluster tree $\mathcal T_p$ and let $A\in \mathbb C^{n\times n}$. Then:
\begin{enumerate}
 \item[(a)] $A(I^\ell_i,I\setminus I^\ell_i)$ is called an HSS block row and $A(I\setminus I^\ell_i,I^\ell_i)$ is called an HSS block column for $i = 1,\ldots,2^{\ell}$, $\ell = 1,\ldots,p$.
 \item[(b)] For $k\in \mathbb N$, $A$ is called a $(\mathcal T_p,k)$-HSS matrix if every HSS block row and column of $A$ has rank at most $k$.
 \item[(c)] The HSS rank of $A$ (with respect to $\mathcal T_p$) is the smallest integer $k$ such that $A$ is a $(\mathcal T_p,k)$-HSS matrix.
\item[(d)] For $k\in \mathbb N$ and $\epsilon > 0$, $A$ is called an $\epsilon$-$(\mathcal T_p,k)$-HSS matrix if every HSS block row and column of $A$ admits an $\epsilon$-approximation of rank $k$. 
\end{enumerate}
\end{definition}

\noindent By Definition~\ref{def:hss} (b), a matrix $A$ with bandwidth $b$ (i.e., $a_{ij} = 0$ for $|i-j|>b$) has HSS rank at most $2b$. 
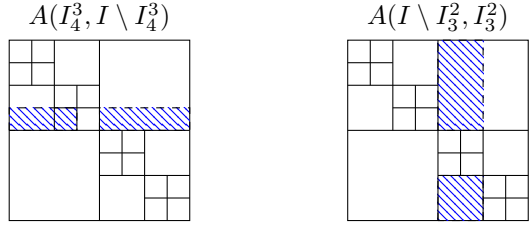
\begin{figure}[!ht]
\centering
\begin{tikzpicture}[scale=0.3]    

      \begin{scope}
        [xshift=0cm]
        \node [above] at (4,8) {$A(I^3_4,I\setminus I^3_4)$};
        \draw (0,0) -- (0,8) -- (8,8) -- (8,0) -- cycle;

        \draw[dashed, pattern=north west lines, pattern color=blue] (0,4) rectangle (3,5);
         \draw[dashed, pattern=north west lines, pattern color=blue] (4,4) rectangle (8,5);
        \draw (6,1) -- (8,1);
        \draw (4,2) -- (8,2);
        \draw (4,3) -- (6,3);
        \draw (0,4) -- (8,4);
        \draw (2,5) -- (4,5);
        \draw (0,6) -- (4,6);
        \draw (0,7) -- (2,7);
        \draw (1,6) -- (1,8);
        \draw (2,4) -- (2,8);
        \draw (3,4) -- (3,6);
        \draw (4,0) -- (4,8);
        \draw (5,2) -- (5,4);
        \draw (6,0) -- (6,4);
        \draw (7,0) -- (7,2);
      \end{scope}
      
      \begin{scope}
        [xshift=15cm]
        \node [above] at (4,8) {$A(I\setminus I^2_3,I^2_3)$};
        \draw (0,0) -- (0,8) -- (8,8) -- (8,0) -- cycle;
        \draw[dashed, pattern=north west lines, pattern color=blue] (4,4) rectangle (6,8);
        \draw[dashed, pattern=north west lines, pattern color=blue] (4,0) rectangle (6,2);
        
        \draw (6,1) -- (8,1);
        \draw (4,2) -- (8,2);
        \draw (4,3) -- (6,3);
        \draw (0,4) -- (8,4);
        \draw (2,5) -- (4,5);
        \draw (0,6) -- (4,6);
        \draw (0,7) -- (2,7);
        \draw (1,6) -- (1,8);
        \draw (2,4) -- (2,8);
        \draw (3,4) -- (3,6);
        \draw (4,0) -- (4,8);
        \draw (5,2) -- (5,4);
        \draw (6,0) -- (6,4);
        \draw (7,0) -- (7,2);
      \end{scope}

\end{tikzpicture}

\caption{ Illustration of an HSS block row and an HSS block column for a cluster tree of depth $3$.}\label{fig:hsscolrow}
\end{figure}

It is intuitive that a matrix satisfying Definition~\ref{def:hss} (d) can be approximated by a $(\mathcal T_p,k)$-HSS matrix with an error of norm proportional to $\epsilon$. Such a result has been shown for the Frobenius norm in~\cite[Corollary 4.3]{Xi2014}. In the following, we show such a result for the matrix 2-norm, with a constant that is tighter than what can be directly concluded from~\cite{Xi2014}. For this purpose, we first introduce some notation and preliminary results. In the following, we say that a block diagonal matrix $D$ conforms with $\mathcal T_p$ if it takes the form
\[
 D = D_1 \oplus D_2 \oplus \cdots \oplus D_{2^p}, \qquad  D_i \in \mathbb C^{n_i^{(p)}\times k_i}
\]
for integers $k_i \le n_i^{(p)}$. In particular, this ensures that the multiplications $D^T A$ and $AD$ do not mix different HSS block rows and columns, respectively.
In the following lemma, we let $\mathcal T^{(k)}_p$ denote the tree associated with the partition $k_1 + k_2 + \cdots k_{2^p}$.
\begin{lemma} \label{lemma:prophss}
Let $A$ be an $\epsilon$-$(\mathcal T_p,k)$-HSS matrix. Then:
\begin{enumerate}
 \item[(a)] $UU^T A$ and $U^T A V$ are $\epsilon$-$(\mathcal T_p,k)$-HSS and $\epsilon$-$(\mathcal T^{(k)}_p,k)$-HSS matrices, respectively, for any block diagonal matrices $U$, $V$ conforming with $\mathcal T_p$ and satisfying $U^T U = V^TV = I$.
 \item[(b)] $A$ is an $\epsilon$-$(\mathcal T_{p-1},k)$-HSS matrix for the tree $\mathcal T_{p-1}$ obtained from $\mathcal T_{p}$ by omitting all leaves.
\end{enumerate}
\end{lemma}
\begin{proof} (a). Consider a node $I^\ell_i$ of $\mathcal T_p$ and the corresponding node $\tilde I^\ell_i$ of $\mathcal T^{(k)}_p$. 

Because of the block diagonal structure of $U$, an HSS block row of $UU^T A$ takes the form $\Pi A(I^\ell_i,I\setminus I^\ell_i)$, where 
$\Pi = U(I^\ell_i,\tilde I^\ell_i) U(I^\ell_i,\tilde I^\ell_i)^T$ is an orthogonal projector.
By assumption, there is a perturbation $C$ with $\|C\|_2 \le \epsilon$ such that $A(I^\ell_i,I\setminus I^\ell_i)+C$ has rank at most $k$. 
In turn, $\Pi A(I^\ell_i,I\setminus I^\ell_i) + \Pi C$ also has rank at most $k$ with $\|\Pi C\|_2\le \|\Pi\|_2 \|C\|_2 = \|C\|_2 \le \epsilon$.
By an analogous argument, the HSS block column $U(I\setminus I^\ell_i,\tilde I\setminus \tilde I^\ell_i) U(I\setminus I^\ell_i,\tilde I\setminus \tilde I^\ell_i)^T A(I\setminus I^\ell_i,I^\ell_i)$  is shown to admit a rank-$k$ approximation of norm $\epsilon$. This proves that $UU^T A$ is an $\epsilon$-$(\mathcal T_p,k)$-HSS matrix.

Now, consider an HSS block row of $U^T A V$ given by $U(I^\ell_i,\tilde I^\ell_i)^T A(I^\ell_i,I\setminus I^\ell_i) V(I\setminus I^\ell_i,\tilde I\setminus \tilde I^\ell_i)$, where both, the left and right factors, have orthonormal columns because of the structure of $U$, $V$. Using the matrix $C$ from above, set \[ \tilde C := U(I^\ell_i,\tilde I^\ell_i)^T C (I\setminus I^\ell_i,\tilde I \setminus \tilde I^\ell_i) V(I\setminus I^\ell_i,\tilde I\setminus \tilde I^\ell_i).\] This perturbation reduces the rank of the HSS block row to at most $k$ and has norm bounded by $\epsilon$ because of the ortogonality of $U,V$. By swapping the roles of $U$ and $V$, the same holds for an HSS block column of $U^T A V$. This proves that $U^T AV$ is an $\epsilon$-$(\mathcal T^{(k)}_p,k)$-HSS matrix.

(b). This part trivially holds because the block rows and columns corresponding to $\mathcal T_{p-1}$ are a subset of the block rows and columns corresponding to $\mathcal T_{p}$.
\end{proof}

\begin{theorem}\label{thm:hss-eps}
 Let $A \in \mathbb C^{n\times n}$ be an $\epsilon$-$(\mathcal T_p,k)$-HSS matrix. Then there exists $\delta A \in \mathbb C^{n\times n}$ with $\|\delta A\|_2 \le \sqrt{2^{p+2}-4} \cdot \epsilon $ such that $A+\delta A$ is a $(\mathcal T_p,k)$-HSS matrix.
\end{theorem}
\begin{proof}
The result is proven by induction on the tree depth $p$. The result trivially holds for $p = 0$.

Let us now consider $p\ge 1$ and suppose that the result holds for any $\epsilon$-$(\mathcal T_{p-1},k)$-HSS matrix and for any tree $\mathcal T_{p-1}$ of depth $p-1$. To establish the result for a tree $\mathcal T_{p}$ of depth $p$, we consider the off-diagonal part $A_{\text{off}}$, that is, $A_{\text{off}}$ is obtained from $A$ by setting the diagonal blocks $A(I_i^p,I_i^p)$ to zero for $i = 1,\ldots,2^p$. This allows us to consider the complete block rows $A_{\text{off}}(I_i^p,I) \in \mathbb C^{n_i^{(p)} \times n}$ instead of the depth-$p$ HSS block rows of $A$. Let
$U_i \in \mathbb C^{n_i^{(p)}\times k}$ contain the $k$ left dominant singular vectors of $A_{\text{off}}(I_i^p,I)$ (if $k\le n_i^{(p)}$, we set $U_i = I_{n_i^{(p)}}$). Because 
$A_{\text{off}}$ is an $\epsilon$-$(\mathcal T_p,k)$-HSS matrix, it holds that
$
\|(I-U_i U_i^T) A_{\text{off}}(I_i^p,I)\|_2 \le \epsilon.
$
The block diagonal matrix $U := U_1 \oplus \cdots \oplus U_{2^p}$ conforms with $\mathcal T_p$ and it is such that
\begin{equation} \label{eq:projectorsleft}
  \|(I-UU^T) A_{\text{off}} \|_2 \le \sqrt{2^p} \epsilon.
\end{equation}
By Lemma~\ref{lemma:prophss}, $UU^T A_{\text{off}}$ is again an $\epsilon$-$(\mathcal T_p,k)$-HSS matrix. This allows us to apply analogous arguments to the depth-$p$ HSS block columns of $UU^T A_{\text{off}}$, yielding a block diagonal matrix $V$ conforming with $\mathcal T_p$ such that
$UU^T  A_{\text{off}} VV^T$ has depth-$p$ HSS block rows/columns of rank at most $k$, and
\[
  \|UU^T  A_{\text{off}}\, (I-VV^T)  \|_2 \le \sqrt{2^p} \epsilon.
\]

Using the notation from Lemma~\ref{lemma:prophss}, $U^T A_{\text{off}} V$ is an $\epsilon$-$(\mathcal T^{(k)}_{p},k)$-HSS matrix and, in turn, an $\epsilon$-$(\mathcal T^{(k)}_{p-1},k)$-HSS matrix. We recall that $\mathcal T^{(k)}_{p-1}$ is the tree of depth $p-1$ obtained by eliminating the leaves of $\mathcal T^{(k)}_{p}$. Hence, the induction hypothesis implies the existence of $\delta_{p-1} A$ such that
$U^T A_{\text{off}} V + \delta_{p-1} A$ is a $(\mathcal T^{(k)}_{p-1},k)$-HSS matrix and
\begin{equation} \label{eq:projectorlevelpm1}
  \| \delta_{p-1} A \|_2 \le \sqrt{2^{p+1}-4} \epsilon.
\end{equation}
The matrix $UU^T A_{\text{off}} VV^T + U \delta_{p-1}A V^T$ is not only a 
$(\mathcal T_{p-1},k)$-HSS matrix but also a
$(\mathcal T_{p},k)$-HSS matrix because, by construction, its depth-$p$ HSS block rows and columns all have rank at most $k$. In summary, $A+\delta A$ is a $(\mathcal T_{p},k)$ matrix with 
\[
 \delta A := -(I-UU^T)A_{\text{off}} - UU^T A_{\text{off}} (I-VV^T) + U \delta_{p-1}A\, V^T.
\]
Exploiting the orthogonality of (co-)ranges and using~\eqref{eq:projectorsleft}--\eqref{eq:projectorlevelpm1}, the norm of this perturbation satisfies
\begin{eqnarray*}
 \|\delta A\|_2^2 & \le & \|(I-UU^T)A_{\text{off}}\|_2^2 + \|UU^T A_{\text{off}} (I-VV^T) + U \delta_{p-1}A\, V^T\|_2^2 \\
 & \le & \|(I-UU^T)A_{\text{off}}\|_2^2 + \|UU^T A_{\text{off}} (I-VV^T)\|_2^2 + \|U \delta_{p-1}A\, V^T\|_2^2 \\
 &\le& 2^p \epsilon^2 + 2^p \epsilon^2 + (2^{p+1}-4) \epsilon^2 = (2^{p+2}-4) \epsilon^2,
\end{eqnarray*}
which completes the proof.
\end{proof}

\subsection{Compressibility of solution $X$ in the HSS format}

We now consider the equation $AX+XB = C$ for $(\mathcal T_p,k)$-HSS coefficients $A,B,C$.
Algorithm~\ref{alg:dac} can be adapted to this situation by simply replacing operations in the HODLR format by operations in the HSS format. In our numerical tests, the HSS compression algorithm from~\cite{Xia2010} is used in line~\ref{step:hodlrcompress}. Moreover, sparse LU factorizations of the matrices $A$ and $B$, are obtained with the MATLAB function \texttt{lu}, in Algorithm~\ref{alg:kryl}. When $A$ and $B$ are non sparse HSS matrices, one can use the algorithms described in \cite{Xia2010} for precomputing either their ULV or Cholesky factorization.

In the following, we show that the solution $X$ can be well approximated by an HSS matrix.
\begin{lemma}\label{lem:hss-sylv}
Let $A,B,C\in\mathbb C^{n\times n}$ be $(\mathcal T_p,k)$-HSS matrices and suppose that $\mathcal W(A)\subseteq E$ and $\mathcal W(-B)\subseteq F$ for sets $E,F\subset \mathbb C$ satisfying $E\cap F=\emptyset$. Let $Y$ be an HSS block row or column of the solution $X$ of \eqref{eq:sylv}. Then
	\[
	\frac{\sigma_{3kh+1}(Y)}{\norm{Y}_2}\leq \big(1+\sqrt 2\big)\cdot Z_{h}(E,F).
	\]
	\end{lemma}  
\begin{proof} We establish the result for an HSS block column 
$X(I\setminus I^\ell_i,I^\ell_i)$; the case of an HSS block row is 
treated analogously. Our proof follows the proof of Theorem~2.7 in
\cite{Massei2017}.

Let us define $A_{11} = A(I^\ell_i, I^\ell_i)$, 
$A_{21} = A(I\setminus I^\ell_i,I^\ell_i)$, 
$A_{12} = A(I^\ell_i, I\setminus I^\ell_i)$, 
$A_{22} = A(I\setminus I^\ell_i, I \setminus I^\ell_i)$, and
$B_{ij}, C_{ij}, X_{ij}$ analogously for $1 \leq i,j \leq 2$. Extracting
the indices $(I\setminus I^\ell_i,I^\ell_i)$ from the equation
$AX+XB = C$, we obtain the relation
\[
 A_{21} X_{11} + A_{22} X_{21} + X_{21} B_{11} + X_{22} B_{21} = C_{21}.
\]
This shows that $X_{21}$ satisfies a Sylvester equation with right-hand side of rank at most $3k$:
\[
 A_{22} X_{21} + X_{21} B_{11} = C_{21} - A_{21} X_{11} - X_{22} B_{21}.
\]
Since $\mathcal W(A_{22})\subseteq \mathcal W(A)$ and 
$\mathcal W(-B_{11})\subseteq \mathcal W(-B)$, 
and $X(I\setminus I^\ell_i,I^\ell_i) = X_{21}$, 
the claim follows from Theorem~\ref{thm:beckermann}. 
	\end{proof}

Combining Lemma~\ref{lem:hss-sylv} with Theorem~\ref{thm:hss-eps} yields the following result.
\begin{corollary} \label{cor:compresshss}
Under the assumptions of Lemma~\ref{lem:hss-sylv}, let $X$ be the solution of \eqref{eq:sylv}. 
Given $\epsilon>0$, let $h$ be the smallest integer that satisfies $(1+\sqrt 2)Z_h(E,F)\leq \frac{\epsilon }{\sqrt{2^{p+2}-4}}$. 
Then there exists a $(\mathcal T_p,3kh)$-HSS matrix $\widetilde X$  such that
$
\norm{X-\widetilde X}_2\leq \epsilon.
$
\end{corollary}

\subsubsection{Complexity of divide-and-conquer in the HSS format}
The complexity analysis of Algorithm~\ref{alg:dac} from Section~\ref{sec:complexity} extends to the HSS format as follows. We retain assumptions (i) and (iii) from Section~\ref{sec:complexity} and replace (ii) and (iv) by:
\begin{enumerate}[label=(\roman*)]
	\item[(ii')] $n=2^ps$ and the input matrices $A,B$ and $C$ are $(\mathcal T_p,k)$-HSS matrices of size $n\times n$;
	\item[(iv')] the compression in Line~\ref{step:hodlrcompress} of Algorithm~\ref{alg:dac} is performed and returns HSS rank $\mathcal O(k)$.
\end{enumerate}
The second part of Assumption (iv') is motivated by the fact that the (exact) matrix $X_0+\delta X$ is the solution of a Sylvester equation with the coefficients satisfying the conditions of Corollary~\ref{cor:compresshss}.  Applied recursively, Assumption (iv') implies that $X_{11}$ and $X_{22}$ have HSS rank $\mathcal O(k)$. Using the fact that matrix-vector multiplications with these matrices have complexity $\mathcal O(kn)$, Line~\ref{line:largerlowrank} requires $\mathcal O(k^2n)$ operations. The LU factorizations of $A$ and $B$ needed in Line~\ref{line:lrupdate} and the compression in Line~\ref{step:hodlrcompress}  have the same complexity~\cite{Xia2010}. Hence, by recursion, the overall complexity of Algorithm~\ref{alg:dac} in the HSS format is $\mathcal O(k^2n\log(n))$.

\subsubsection{Reducing the rank of updates in the Hermitian case}

The splitting~\eqref{eq:hodlrsplit}, the basis of our divide-and-conquer method, leads to  perturbations of rank $2k$ for general $(\mathcal T_p,k)$-HODLR and HSS matrices. 
For a Hermitian positive definite $(\mathcal T_p,k)$-HSS matrix $A$, let $A_{21} = U\Sigma V^*$ be the singular value decomposition of the subdiagonal block on level $1$. Instead of~\eqref{eq:hodlrsplit} we then consider the splitting
\begin{equation} \label{eq:defA}
   A = \begin{bmatrix}
    A_{11} & V \Sigma U^* \\
    U\Sigma V^*   & A_{22} \\
  \end{bmatrix} = A_0 + \delta A :=  
  \begin{bmatrix}
    A_{11} + V\Sigma V^* & 0 \\
    0 & A_{22} +  U\Sigma U^* \\
  \end{bmatrix}
  + \begin{bmatrix}
     V \\ -U
  \end{bmatrix}\Sigma \begin{bmatrix}
    -V \\  U
  \end{bmatrix}^*.
\end{equation}
The obvious advantage is that the perturbation now has rank $k$. However, in order to be a useful basis for divide-and-conquer method, $A_0$ needs to inherit the favorable properties of $A$. This is shown by the following lemma.
\begin{lemma} \label{lem:update-hss}
	Let $A$ be a Hermitian positive definite $(\mathcal T_p,k)$-HSS matrix, partitioned as in~\eqref{eq:defA}. Then $A_0$ is also a Hermitian positive definite $(\mathcal T_p,k)$-HSS matrix. 
\end{lemma}

\begin{proof}
Note that \[ A_{11} + V\Sigma V^* = \begin{bmatrix} A_{11} & V \Sigma U^* \end{bmatrix} \begin{bmatrix} I \\ UV^* \end{bmatrix} = \begin{bmatrix} I & V U^* \end{bmatrix} \begin{bmatrix}
    A_{11}  \\
    U\Sigma V^*    \\
  \end{bmatrix}. \] The first relation implies that the rank of an HSS block row of $A_{11} + V\Sigma V^*$ is bounded by the rank of the corresponding HSS block row of $\begin{bmatrix} A_{11} & V \Sigma U^* \end{bmatrix}$, which is bounded by $k$. The second relation implies that the rank of an HSS block column of $A_{11} + V\Sigma V^*$ is bounded by the rank of the corresponding HSS block column of $\begin{bmatrix}
    A_{11}  \\
    U\Sigma V^*    \\
  \end{bmatrix}$, which is also bounded by $k$. An analogous argument applies to the HSS block rows and columns of $A_{22} +  U\Sigma U^*$. Thus, $A_0$ is a $(\mathcal T_p,k)$-HSS matrix.
  It is straightforward to verify that $A_0$ is Hermitian positive definite.
\end{proof} 
The right hand side $C = C_0 +\delta C$ in \eqref{eq:sylv} is treated similarly, for lowering the rank of the right hand side of \eqref{eq:delta-sylv}. Since we do not need to preserve any positive definiteness in $C_0$, we are allowed to choose $\delta C= \begin{bmatrix}
\theta V \\ -U
\end{bmatrix}\Sigma \begin{bmatrix}
-V \\  \theta^{-1}U
\end{bmatrix}^*$, for $\theta\neq 0$.

	\begin{remark}
		In the special case when $A$ is a Hermitian banded matrix, with bandwidth smaller than $s$, the updates $VV^*$ and $UU^*$ only affect the smallest diagonal blocks in the south east corner of $A_{11}$ and in the north-west corner of $A_{22}$, respectively. In particular, the sparsity pattern of the off-diagonal blocks is maintained.
		\end{remark}

%% file: section5.tex
\section{Numerical results}\label{sec:num-res}


In this section, we illustrate the performance of our divide-and-conquer method from Section~\ref{sec:dac} for a number of different examples. In particular, we consider linear matrix equations $AX+XB = C$ for which $A, B, C$ are efficiently representable as HSS or HODLR matrices. We exclude cases where $C$ has low rank (or low numerical rank), since these can be treated more efficiently directly by ADI or Krylov subspace methods.

We compare our method  with other techniques developed for linear matrix equations with rank-structured  coefficients. In particular, this includes the matrix sign function iteration for HODLR matrices
proposed in \cite{Grasedyck2003a}, and recently tested in \cite{Massei2017}. When the coefficients $A$ and $B$ are symmetric positive definite, well conditioned and banded, we also compare with the approach proposed in \cite{Palitta2017}, a matrix version of the conjugate gradient that exploits the approximate sparsity in the solution. 

A number of approaches are based on applying numerical quadrature to $X = \int_{0}^\infty e^{-tA} C e^{-tB} dt$; for example in~\cite{Haber2016} in the context of sparsity and in~\cite{Massei2017} in the context of HODLR matrices. As demonstrated in~\cite{Massei2017} such approaches are less competitive
compared to the sign iteration and they are therefore not included in our comparison.

\subsection{Details of the implementation}\label{sec:details}

All experiments have been performed on a Laptop with the dual-core Intel Core i7-7500U 2.70 GHz CPU, 256KB of level 2 cache, and 16 GB of RAM. The
algorithms are implemented in MATLAB and tested under MATLAB2016a, 
with MKL BLAS version 11.2.3 utilizing both cores.

The methods described in Section~\ref{sec:alg-and-ric} and Section~\ref{sec:dac} require the choice of several parameters:
\begin{itemize}
	\item $\tau_{\mathsf{NW}} =$ tolerance for stopping the Newton method, see Line~\ref{line:stopnewton} of Algorithm~\ref{alg:riccati};
	\item $\tau_{\mathsf{EK}} = $ tolerance for stopping the extended Krylov subspace method, see~\eqref{eq:stop};
	\item $s = $ size of the diagonal blocks in the HODLR/HSS block partitioning;
	\item $\tau_{\sigma} = $ tolerance for low-rank truncations when compressing the right-hand side (see Section~\ref{sec:compress}), the output of Algorithm~\ref{alg:kryl}, as well as HODLR and HSS matrices in Line~\ref{step:hodlrcompress} of Algorithm~\ref{alg:dac}.
\end{itemize}
 Concerning the compression in the HODLR and HSS formats, we specify that in each off-diagonal block we discard the singular values that are relatively small with respect to the norm of the whole matrix, that can be cheaply estimated with a few steps of the power method. 

The values of the parameters used in the various experiments are reported in Table~\ref{tab:param}.
We have found that the performance of the proposed algorithms is not very sensitive
to the choices of the tolerances reported in Table~\ref{tab:param}: smaller tolerances lead to
more accurate results, as one would expect. It is, however, advisable to choose $\tau_{\mathsf{EK}}$ and $\tau_\sigma$ on a similar level, in order to avoid wasting
computational resources. The tolerance $\tau_{\mathsf{NW}}$ can be chosen larger because the quadratic convergence of the Newton method implies that the actual error $\norm{X_{k+1} - X^*}$ is proportional to $\tau_{\mathsf{NW}}^2$.


To assess the accuracy of an approximate solution $\hat X$ of a Sylvester equation, we report the  residual 
$\mathrm{Res}(\hat X) = \norm{A\hat X+\hat XB-C}_2 / ((\norm{A}_2+\norm{B}_2)\norm{\hat X}_2)$
which is linked to the relative backward error on the associated linear system~\cite{Higham2002}. For CAREs we consider the quantity 
$\mathrm{Res}(\hat X) = \norm{A\hat X+\hat XA^*-\hat XB\hat X-C}_2 / \norm{A X_0+X_0A^*-X_0BX_0-C}_2$ instead.

\begin{table}\label{tab:param}
\centering
\label{my-label}
\begin{tabular}{l|llll}
	Test&$\tau_{\mathsf{NW}}$  &$\tau_{\mathsf{EK}}$   & $s$  &$\tau_{\sigma}$  \\ \hline
	Example~\ref{exa:1}&$10^{-8}$  &$10^{-8}$   & - &-  \\ 
	Example~\ref{exa:2}&$10^{-8}$  &$10^{-12}$  &-  &-  \\ 
	Section~\ref{sec:laplacian}&-  &$10^{-12}$   &$256$  &$10^{-12}$  \\ 
	Section~\ref{sec:conv-diff}&-  &$10^{-12}$   &$256$  &$10^{-12}$  \\ 
	Section~\ref{sec:heat}&-  &$10^{-6}$   &$256$  &$10^{-6}$  \\ 
	Section~\ref{sec:ric-band}&$10^{-8}$  &$10^{-12}$   &$256$  &$10^{-12}$  \\ 
	Section~\ref{sec:shuffle}&$10^{-8}$   &$10^{-12}$    &$256$  &$10^{-12}$\\
	\hline
\end{tabular}
\caption{Choices of parameters used in the experiments.}
\end{table}

\subsection{Discretized 2D Laplace equation} \label{sec:laplacian}
We consider the two-dimensional Laplace equation 
	\begin{equation} \label{eq:laplace2d}
	   \begin{cases}
	     - \Delta u = f(x,y) & (x,y) \in \Omega \\
	     u(x,y) = 0 & (x,y) \in \partial \Omega \\
	   \end{cases}, \qquad 
	    \Delta u = \frac{\partial^2 u}{\partial x^2} + 
	   \frac{\partial^2 u}{\partial y^2}, 
	\end{equation}
	for the square domain $\Omega = [ 0, 1 ]^2$ and $f(x,y) = \log(1 + |x - y|)$.  It is well known that the central finite 
	difference discretization~\eqref{eq:laplace2d} on a regular grid leads to a Lyapunov equation $AX + XA = C$ with coefficients
	$A = (n+1)^2\cdot \trid(-1, 2, -1)$ and $C$ containing  samples of $f(x,y)$ on the grid.  The latter matrix does not have low (numerical) rank, but it can be well approximated in the HODLR and HSS formats relying on the Chebyshev expansion of $f$ in the off-diagonal sub domains of $\Omega$, see the discussion in \cite[Example 6.1]{Massei2017}.
	
Table~\ref{tab:laplacian} and Figure~\ref{fig:laplacian} compare the performance of the matrix sign function iteration in the HODLR format with the divide-and-conquer method in both, the HODLR and HSS formats. %
	\begin{table}[t]
		\centering
		
		\small
		\begin{filecontents*}{hss_experiment1.dat}
512		0.56808	9.0435e-13	21	0.4156	4.3154e-13	20	0.31955	6.7088e-13	20
1024		1.8456	1.5971e-12	22	1.162	7.7035e-13	21	1.3797	7.3604e-13	21
2048		5.4546	2.0936e-12	23	3.1892	7.506e-13	22	3.3275	9.8598e-13	22
4096		15.217	3.3863e-12	22	7.4516	6.8522e-13	23	8.3237	8.0326e-13	23
8192		40.483	5.6922e-12	23	16.889	8.0103e-13	24	16.845	7.4742e-13	24
16384		97.297	7.0448e-12	24	39.871	6.8416e-13	26	38.585	7.3721e-13	26
32768		242.96	8.9398e-12	23	86.275	7.0796e-13	28	88.267	8.8917e-13	29
65536		591.21	1.0119e-11	21	189.73	6.4521e-13	33	184.94	9.642e-13	33
1.3107e+05	1433.9	9.9931e-12	22	393.95	7.1003e-13	33	377.44	1.0636e-12	35
\end{filecontents*}
\pgfplotstabletypeset[%
		every last row/.style={after row=\bottomrule},
		sci zerofill,
		columns={0,1,2,4,5,7,8},
		columns/0/.style={column name=$n$},
		columns/1/.style={column name=$\mathrm{T}_{\text{Sign}}$},
		columns/2/.style={column name=$\mathrm{Res}_{\text{Sign}}$},    
		columns/4/.style={column name=$\mathrm{T}_{\text{D\&C HODLR}}$},
		columns/5/.style={column name=$\mathrm{Res}_{\text{D\&C HODLR}}$},
		columns/7/.style={column name=$\mathrm{T}_{\text{D\&C HSS}}$},
		columns/8/.style={column name=$\mathrm{Res}_{\text{D\&C HSS}}$}
		]{hss_experiment1.dat}
		\caption{Execution times (in seconds) and relative residuals for the matrix sign function iteration and the divide-and-conquer method applied to the discretized 2D Laplace equation from Section~\ref{sec:laplacian}.}
		\label{tab:laplacian}
	\end{table}
	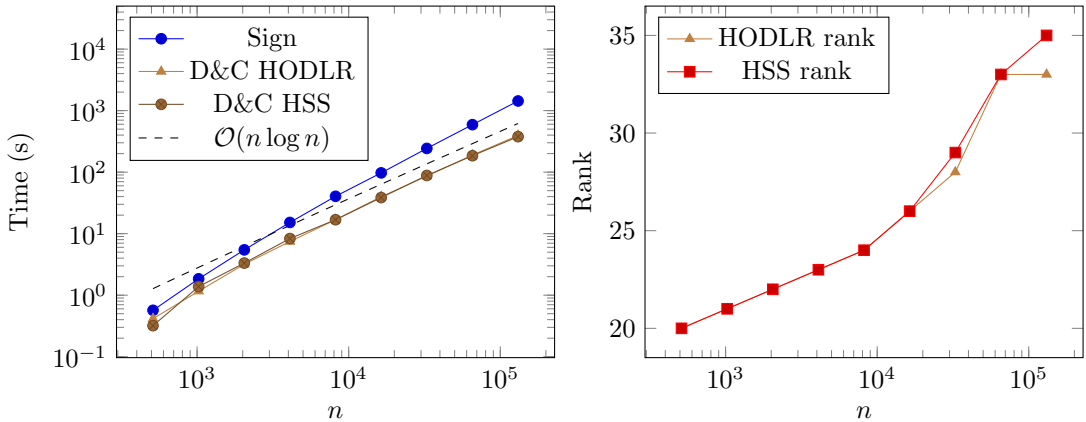
\begin{figure}
	\centering
	
		\begin{tikzpicture}
		\begin{loglogaxis}[width=.49\linewidth, height=.32\textheight,
		legend pos = north west,
		xlabel = $n$, ylabel = Time (s), ymax= 5e4]
		\addplot table[x index=0, y index=1]  {hss_experiment1.dat};
		\addplot[brown,mark=triangle*] table[x index=0, y index=4]  {hss_experiment1.dat};
		\addplot table[x index=0, y index=7] {hss_experiment1.dat};
		\addplot[domain=512:131000,dashed] {4e-4 * x * ln(x) };
		\legend{Sign, D\&C HODLR,D\&C HSS, $\mathcal O(n \log n)$};
		\end{loglogaxis}
		\end{tikzpicture}	
		\begin{tikzpicture}
		\begin{semilogxaxis}[width=.49\linewidth, height=.32\textheight,
		legend pos = north west,
		xlabel = $n$, ylabel = Rank]
		\addplot[brown,mark=triangle*] table[x index=0, y index=6]  {hss_experiment1.dat};
		\addplot table[x index=0, y index=9] {hss_experiment1.dat};
		\legend{HODLR rank,HSS rank};
		\end{semilogxaxis}
		\end{tikzpicture}
		
		\caption{On the left, timings of the algorithms applied to the Laplacian equation with respect to the grid size.
			The dashed lines report the expected theoretical complexity of the divide-and-conquer strategies. On the right, HODLR and HSS rank of the solutions returned by the divide-and-conquer methods.}
		\label{fig:laplacian}
	\end{figure}

The divide-and-conquer method is based on extended Krylov subspaces with principal submatrices of $A$. As these matrices inherit the tridiagonal structure of $A$, they can be easily applied and inverted. In contrast, the matrix sign iteration method does not preserve bandedness and needs to operate with general HODLR matrices. This is significantly less efficient; in turn, our divide-and-conquer method is always faster and scales more favorably as $n$ increases. Moving from the HODLR to the HSS format results in further (modest) speedup. The HODLR and HSS ranks remain reasonably small in all approaches.  
One major advantage of the HSS format is its reduced memory requirements; for example for $n = 1.31\cdot 10^5$, $433$ MByte and $267$ MByte are required to store the approximate solution in the HODLR and HSS formats, respectively. 
	
\subsection{Convection diffusion}\label{sec:conv-diff}
We repeat the experiment from Section~\ref{sec:laplacian} for the convection-diffusion equation
\[
\begin{cases}
- \Delta u + v\nabla u = f(x,y) & (x,y) \in \Omega:=[0,1] \\
u(x,y) = 0 & (x,y) \in \partial \Omega \\
\end{cases},
\]
where $v=[10,10]$ and, once again, $f(x,y) = \log(1 + |x - y|)$. A standard finite difference discretization now leads to a Lyapunov equation $AX+XA^T=C$ with the nonsymmetric matrix 
\[
A=(n+1)^2 \begin{bmatrix}
2&-1\\
-1&2&-1\\
&\ddots&\ddots&\ddots\\
&&-1&2&-1\\
&&&-1&2
\end{bmatrix}+\frac{5}{2}(n+1)
\begin{bmatrix}
3&-5&1\\
1&3&-5&\ddots\\
&\ddots&\ddots&\ddots&1\\
&&1&3&-5\\
&&&1&3
\end{bmatrix}
\]  
 and $C$ as in Section~\ref{sec:laplacian}. 
Table~\ref{tab:conv-diff} displays the timings and the relative residuals obtained for this example, reconfirming our observations for the symmetric example from Section~\ref{sec:laplacian}. Also, we have observed the HODLR and HSS ranks to behave in a similar manner. 

\begin{table}[t]
	\centering
	
	\small
	\begin{filecontents*}{hss_experiment3.dat}
512		0.60267	1.1452e-12	20	0.42143	4.853e-13	20	0.38193	4.8545e-13	20
1024		1.8713	9.5135e-13	25	1.4655	6.589e-13	21	1.4664	7.655e-13	21
2048		5.4182	1.78e-12	24	3.2169	4.5061e-13	23	3.4769	5.5744e-13	23
4096		15.822	2.9427e-12	26	8.3735	4.6231e-13	24	8.0766	7.6367e-13	24
8192		40.153	4.3782e-12	25	21.483	7.5581e-13	25	19.642	6.4688e-13	25
16384		99.175	6.5248e-12	27	40.413	6.2346e-13	27	40.991	8.8334e-13	26
32768		236.81	8.1157e-12	25	94.899	8.302e-13	28	89.718	6.9583e-13	28
65536		603.59	8.3736e-12	25	198.04	8.6301e-13	33	207.62	8.0568e-13	34
1.3107e+05	1365.6	7.8513e-12	24	430.47	8.5162e-13	33	418.08	8.4316e-13	32
\end{filecontents*}
\pgfplotstabletypeset[%
	every last row/.style={after row=\bottomrule},
	sci zerofill,
	columns={0,1,2,4,5,7,8},
	columns/0/.style={column name=$n$},
	columns/1/.style={column name=$\mathrm{T}_{\text{Sign}}$},
	columns/2/.style={column name=$\mathrm{Res}_{\text{Sign}}$},    
	columns/4/.style={column name=$\mathrm{T}_{\text{D\&C HODLR}}$},
	columns/5/.style={column name=$\mathrm{Res}_{\text{D\&C HODLR}}$},
	columns/7/.style={column name=$\mathrm{T}_{\text{D\&C HSS}}$},
	columns/8/.style={column name=$\mathrm{Res}_{\text{D\&C HSS}}$},
	columns/9/.style={column name=$ HSS_{rank}$}
	]{hss_experiment3.dat}
	\caption{Execution times (in seconds) and relative residuals for the matrix sign function iteration and the divide-and-conquer method applied to the discretized convection-diffusion equation from Section~\ref{sec:conv-diff}.}
	\label{tab:conv-diff}
\end{table}
    \subsection{Heat equation}\label{sec:heat}
    
    A model describing
    the temperature change of a thermally actuated deformable
    mirror used in extreme ultraviolet litography~\cite[Section 5.1]{Haber2016} leads to a symmetric Lyapunov equation $AX + XA = C$ with coefficients 
    \begin{align*}
      A &= I_{q} \otimes \trid_6(b,a,b) + 
       \trid_q(b,0,b) \otimes I_6\\
      C&= I_{q} \otimes (-c \cdot E_6 + (c-1) \cdot I_6)
        + \trid_q(d,0,d) \otimes E_6, 
    \end{align*}
    where $a=-1.36$, $b=0.34$, $c=0.2$, $d=0.1$ 
    and $E_6$ is the $6 \times 6$ matrix of all ones. Note that $A$ and $C$ are banded with bandwidth $6$ and $11$, respectively. As analyzed in~\cite[Ex. 2.6]{Palitta2017}, the condition number of $A$ is bounded by $40$ and hence the sparse conjugate gradient (CG) method proposed there scales linearly with $n:=6q$. We executed the sparse CG setting $X_0:=0$ and using  $\norm{AX_k+X_kA-C}_F/\norm{C}_F\leq 10^{-6}$, for stopping the iterations, as suggested in \cite{Palitta2017}. Although it is faster and its advantageous scaling is clearly visible, approximate sparsity is, compared to the HODLR and HSS format, significantly less effective at compressing the solution $X$; see Figure~\ref{fig:heat}. 
    The observed HODLR and HSS ranks are equal to $10$ and $20$ respectively, independently of $n$. 
    	\begin{table}[t]
    	\centering
    	
    	\small
    	\begin{filecontents*}{hss_experiment2.dat}
1536	0.86246	5.9522e-08	11952		0.85626	1.2257e-08	10	3025.8		0.75636	1.2285e-08	20	2796.3
3072	1.7746	5.8584e-08	25068		2.0297	1.2365e-08	10	6531.6		2.5028	1.229e-08	20	5710.5
6144	4.3906	5.7622e-08	51300		4.6187	1.2367e-08	10	14023		4.6061	1.2346e-08	20	11561
12288	8.0779	5.7083e-08	1.0376e+05	11.788	1.2379e-08	10	29967		10.983	1.2341e-08	20	23284
24576	17.062	5.6804e-08	2.0869e+05	22.722	1.2319e-08	10	63774		24.61	1.2311e-08	20	46751
49152	35.549	5.6663e-08	4.1855e+05	58.243	1.2304e-08	10	1.3523e+05	53.096	1.2331e-08	20	93707
98304	71.134	5.6592e-08	8.3826e+05	128.28	1.2334e-08	10	2.8582e+05	125.32	1.2324e-08	20	1.8764e+05
\end{filecontents*}
\pgfplotstabletypeset[%
    	every last row/.style={after row=\bottomrule},
    	sci zerofill,
    	columns={0,1,2,4,5,8,9},
    	columns/0/.style={column name=$n$},
    	columns/1/.style={column name=$\mathrm{T}_{\text{CG}}$},
    	columns/2/.style={column name=$\mathrm{Res}_{\text{CG}}$},    
    	columns/4/.style={column name=$\mathrm{T}_{\text{D\&C HODLR}}$},
    	columns/5/.style={column name=$\mathrm{Res}_{\text{D\&C HODLR}}$},
    	columns/8/.style={column name=$\mathrm{T}_{\text{D\&C HSS}}$},
    	columns/9/.style={column name=$\mathrm{Res}_{\text{D\&C HSS}}$}
    	]{hss_experiment2.dat}
    	\caption{Execution times (in seconds) and relative residuals for the sparse CG method and the divide-and-conquer method applied to the head equation from Section~\ref{sec:heat}.  
    		}
    	\label{tab:heat}
    \end{table}
    
    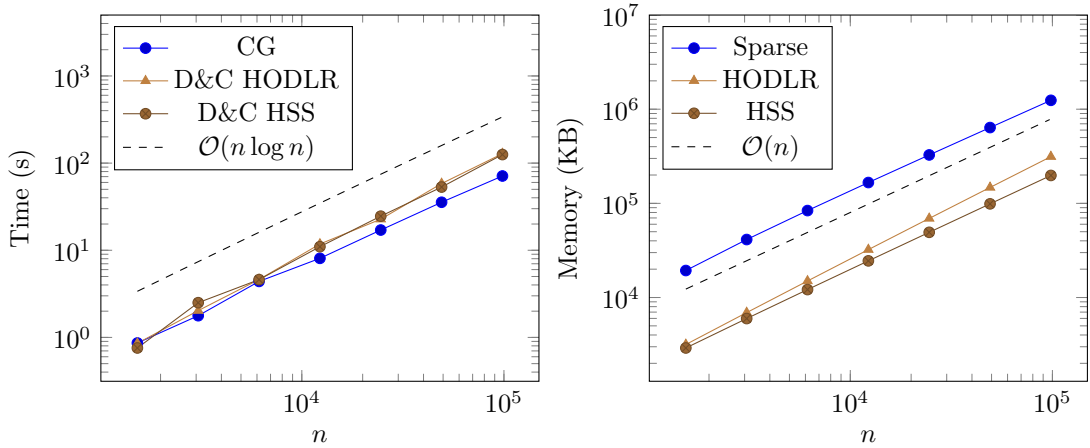
\begin{figure}
    	\centering
    	
    	\begin{tikzpicture}
    	\begin{loglogaxis}[width=.49\linewidth, height=.33\textheight,
    	legend pos = north west,
    	xlabel = $n$, ylabel = Time (s), ymax = 5e3]
    	\addplot table[x index=0, y index=1]  {hss_experiment2.dat};
    	\addplot[brown,mark=triangle*] table[x index=0, y index=4]  {hss_experiment2.dat};
    	\addplot table[x index=0, y index=8] {hss_experiment2.dat};
    	\addplot[domain=1536:97000,dashed] {3e-4 * x * ln(x) };
    	\legend{CG, D\&C HODLR,D\&C HSS, $\mathcal O(n \log n)$};
    	\end{loglogaxis}		
    	\end{tikzpicture}
    	\begin{tikzpicture}
    	\begin{loglogaxis}[width=.49\linewidth, height=.33\textheight,
    	legend pos = north west,
    	xlabel = $n$, ylabel = Memory (KB), ymax = 10^7]
    	\addplot table[x index=0, y index=1]  {memory.dat};
    	\addplot[brown,mark=triangle*] table[x index=0, y index=2]  {memory.dat};
    	\addplot table[x index=0, y index=3] {memory.dat};
    	\addplot[domain=1536:97000,dashed] {8* x  };
    	\legend{Sparse, HODLR,HSS, $\mathcal O(n)$};
    	\end{loglogaxis}
    	\end{tikzpicture}
    	\caption{Time (left) and memory (right) consumptions for solving  the Heat equation with
    		different grid sizes.
    		The dashed lines report the expected asymptotic complexity for the divide-and-conquer with HSS matrices.}
    	\label{fig:heat}
    \end{figure}

    \subsection{A large-scale Riccati equation with banded and low-rank coefficients}\label{sec:ric-band}
    
In this example, we demonstrate how a combination of 
Algorithms~\ref{alg:riccati} and~\ref{alg:dac} can be used to address certain large scale Riccati equations. Consider the CARE $AX + XA^* - XBX - C = 0$ with the coefficients
\begin{align*}
A=\trid_n(1,-2,1)\in\mathbb R^{n\times n},\quad  B=B_UB_U^T,\quad B_U=\begin{bmatrix} e_1 & e_n
\end{bmatrix}\in\mathbb R^{n\times 2},\quad C=-I_n.
\end{align*}
As the matrix $A$ is negative definite, we can choose $X_0=0$ as the stabilizing initial guess in the Newton method. In turn, the Lyapunov equation in the first step (see  Line~\ref{line:firstlyapnewton} of Algorithm~\ref{alg:riccati}) takes the form $AX+XA=C$. Exploiting the structure of the coefficients, we address this equation with Algorithm~\ref{alg:dac} in the HSS format. For all subsequent iterations, we use the low-rank update procedure described in Section~\ref{sec:riccati}, recompressing the intermediate solution $X_k$ in the HSS format.

In contrast to the observations made in Section~\ref{sec:ric-test}, the results displayed 
in Table~\ref{tab:exp3} now reveal that the first step does not dominate the cost of the overall algorithm. Note that $T_{\text{avg}}$, the average time per Newton step, grows more than linearly as $n$ increases, due to the fact that the condition number of $A$ increases and, in turn, the extended Krylov subspace method converges more slowly.
As $n$ increases, the norm of the final solution grows accordingly to the final residue.
The HSS rank of the approximate solution $X$ grows  slowly, apparently only logarithmically with $n$. 
		
\begin{table}[t]
	\centering 		
	\small
	\begin{filecontents*}{e1.dat}
1024		0.78935	0.2824	0.35776	0.063369	2.5813e-07	9	14	31090
2048		1.6284	0.57493	0.35306	0.11705		1.2908e-06	10	16	1.2412e+05
4096		3.5654	1.2141	0.34053	0.26125		6.5491e-06	10	19	4.9598e+05
8192		8.3367	2.5808	0.30958	0.57558		3.1032e-05	11	21	1.983e+06
16384		21.388	6.1291	0.28657	1.5259		0.00011047	11	24	7.93e+06
32768		55.534	11.835	0.21311	3.9726		0.00053158	12	25	3.1716e+07
65536		170.61	26.799	0.15708	13.073		0.0028184	12	29	1.2686e+08
1.3107e+05	475.11	54.231	0.11414	35.073		0.023553	13	31	5.0771e+08
\end{filecontents*}
\pgfplotstabletypeset[%
	every last row/.style={after row=\bottomrule},		
	sci zerofill,
	columns={0,8,1,3,4,5,6,7},
	columns/0/.style={column name=$n$},
	columns/8/.style={column name=$\norm{\hat X}_2$, sci},
	columns/1/.style={column name=$\mathrm{T}_{\text{tot}}$},
	columns/3/.style={column name=$\frac{T_{\texttt{step 1}}}{T_{\text{tot}}}$,fixed},
	columns/4/.style={column name=$\mathrm{T}_{\text{avg}}$,fixed},    
	columns/5/.style={column name=Res},
	columns/6/.style={column name=it},
	columns/7/.style={column name=HSS rank}
	]{e1.dat}
	\caption{Performance of Algorithm~\ref{alg:riccati}, combined with Algorithm~\ref{alg:dac} in the HSS format for the first step, for the CARE from Section~\ref{sec:ric-band}.}
	\label{tab:exp3}
\end{table}

\subsection{A large-scale Riccati equation from a second-order problem}\label{sec:shuffle}

Let us consider a linear second-order control system
\[
  M \ddot{z} + L \dot z = Kz = Du,
\]
where $M$ is diagonal and $K$, $L$ are banded (or, more generally, HSS).
Applying linear-optimal control leads, after a suitable linearization, to a CARE~\eqref{eq:riccati} with the matrix $A$ taking the form\footnote{Note that, in contrast to~\cite{Abels1999}, we use the second companion linearization in order to be consistent with our choice of transposes in~\eqref{eq:riccati}.}
\begin{equation}  \label{eq:matrixA}
	A = \begin{bmatrix}
	  0 & -M^{-1} K \\
	  I_q & -M^{-1} L 
	\end{bmatrix}.
\end{equation}
In fact, the matrix $A$ from Example~\ref{exa:2} is of this type, with $K$ tridiagonal and $M,L$ (scaled) identity matrices. It turns out that $A$ does \emph{not} have low HODLR or HSS rank. In the following, we explain a simple trick to turn $A$ into an HSS matrix, which then allows us to apply and the techniques from Section~\ref{sec:ric-band} to Example~\ref{exa:2}.

We first observe that the matrix $A$ from~\eqref{eq:matrixA} can be decomposed as
\[
  A = \begin{bmatrix} 0 & 0 \\ 1 & 0 \end{bmatrix} \otimes I_q - 
    \begin{bmatrix} 0 & 1 \\ 0 & 0 \end{bmatrix} \otimes M^{-1} K - 
    \begin{bmatrix} 0 & 0 \\ 0 & 1 \end{bmatrix} \otimes M^{-1} L. 
\]
Let $\Pi$ denote the perfect shuffle permutation~\cite{VanLoan2000}, which swaps the order
in the Kronecker product of matrices of sizes $2$ and $q$: $\Pi (X \otimes Y) \Pi^* = Y \otimes X$, for any
$X \in \mathbb{C}^{2 \times 2}$, $Y \in \mathbb{C}^{q \times q}$.  
Hence,
\begin{equation} \label{eq:tildeA}
  \widetilde A := \Pi A \Pi^* = I_q \otimes \begin{bmatrix} 0 & 0 \\ 1 & 0 \end{bmatrix}
    -M^{-1} K \otimes \begin{bmatrix} 0 & 1 \\ 0 & 0 \end{bmatrix}
    -M^{-1} L \otimes \begin{bmatrix} 0 & 0 \\ 0 & 1 \end{bmatrix}. 
\end{equation}
The following result allows us to control the HSS ranks for each of the terms.
\begin{lemma}\label{lemma:hsskron}
	Let $A\in\mathbb C^{q\times q}$ 
	be an $(\mathcal T_p,k_A)$-HSS matrix, and let $B \in \mathbb{C}^{m \times m}$ have rank $k_B$. Then $A \otimes B$ is a $(\mathcal T_p^{(mq)}, k_Ak_B)$-HSS matrix, where $\mathcal T_p^{(mq)}$ is the cluster tree defined by the integer partition 
	\[ mq = mq_1 + mq_2 + \cdots + mq_{2^p}. \] 
\end{lemma}
\begin{proof} 
	The results follows immediately considering that an HSS
	block row $\hat X$ in $A$ corresponds to the HSS block row
	$\hat X \otimes B$ in $A \otimes B$, with respect to $\mathcal T_p^{(m)}$. If the former has rank bounded
	by $k_A$, then the latter has rank bounded by $k_Ak_B$. Analogously for HSS block columns.
\end{proof}

Lemma~\ref{lemma:hsskron} implies that the matrix $\widetilde A$ from~\eqref{eq:tildeA} is a 
$(\mathcal T_p^{(2)},k_1 + k_2)$-HSS if $M^{-1} K$ and $M^{-1} L$ are 
$(\mathcal T_p,k_1)$- and $(\mathcal T_p,k_2)$-HSS matrices, respectively.
For Example~\ref{exa:2} these assumptions are satisfied
with $k_1 = 0, k_2 = 1$. In turn this allows us to apply the techniques from Section~\ref{sec:ric-band} to the shuffled Riccati equation from Example~\ref{exa:2}, using the shuffled starting guess $\Pi X_0\Pi^*$. Table~\ref{tab:shuffle} displays the obtained results. We highlight that the non-symmetric Lyapunov equation that is solved in the first step of the Newton method does not satisfy the hypotheses of Lemma~\ref{thm:uniform-approximability}. In fact, the field of values of the matrix $A-X_0B$ is not contained in the open left half plane. Still, Algorithm~\ref{alg:dac} is observed to perform very well.

\begin{table}[t]
	\centering 		
	\small
	\begin{filecontents*}{e4.dat}
512	0.87794	0.45843	0.52216	0.052439	2.901e-08	9	13	15489
1024	1.1384	0.4536	0.39845	0.076088	1.3429e-07	10	13	61942
2048	2.003	0.77764	0.38824	0.13615	6.4415e-07	10	16	2.4776e+05
4096	4.5022	1.6053	0.35656	0.28969	2.4127e-06	11	19	9.9101e+05
8192	10.008	3.3347	0.33322	0.66728	1.0028e-05	11	20	3.964e+06
16384	23.818	6.867	0.28831	1.541	7.3278e-05	12	23	1.5856e+07
32768	60.673	14.615	0.24087	4.1871	0.00033577	12	24	6.3423e+07
\end{filecontents*}
\pgfplotstabletypeset[%
	every last row/.style={after row=\bottomrule},			
	sci zerofill,
	columns={0,8,1,3,4,5,6,7},
	columns/0/.style={column name=$n$},
	columns/8/.style={column name=$\norm{\hat X}_2$, sci},
	columns/1/.style={column name=$\mathrm{T}_{\text{tot}}$},
	columns/3/.style={column name=$\frac{T_{\texttt{step 1}}}{\mathrm{T}_{\text{tot}}}$,fixed},
	columns/4/.style={column name=$\mathrm{T}_{\text{avg}}$,fixed},
	columns/5/.style={column name=Res},
	columns/6/.style={column name=it},
	columns/7/.style={column name=HSS rank}
	]{e4.dat}
	\caption{Performance of Algorithm~\ref{alg:riccati}, combined with Algorithm~\ref{alg:dac} in the HSS format for the first step, for the (shuffled) CARE from Section~\ref{sec:shuffle}.
}
	\label{tab:shuffle}
\end{table}

%% file: conclusions.tex
\section{Concluding remarks}\label{sec:conclusion}
	
%

We have proposed a Krylov subspace method for updating the solution of linear matrix equations whose coefficients are affected by low-rank perturbations. We have shown that our approach can significantly speed up the Newton iteration for solving certain CAREs.  
Moreover, we have designed a divide-and-conquer algorithm for linear matrix equations with hierarchically low-rank coefficients. A theoretical analysis of the structure preservation and of the computational cost has been provided. In the numerical tests, we have verified that our algorithm scales well with the size of the problem and often outperforms existing techniques that rely on approximate sparsity and data sparsity. 

During this work, we encountered two issues that might deserve further investigation. 1. The structure of a stable Lyapunov equation is currently exploited only partially; see Section~\ref{sec:stable}. In particular, it is an open problem to design a divide-and-conquer method that aims directly at the Cholesky factor of the solution and thus preserves its semi-definiteness. 2. As seen in Section~\ref{sec:heat}, it can be advantageous to exploit (approximate) sparsity in the case of well-conditioned equations. It would be interesting to design a variant of the divide-and-conquer method that benefits from sparsity as well.

\begin{paragraph}{Acknowledgments} We thank Jianlin Xia for pointing out~\cite{Xia2012}, and the referees for their careful reading and helpful remarks. \end{paragraph}